\documentclass[reqno,  12pt]{amsart}
\usepackage{extarrows}
\usepackage{ragged2e}
\usepackage{amssymb}
\usepackage{amsfonts}
\usepackage{amsmath}
\usepackage{mathrsfs}
\usepackage{diagbox}
\usepackage{color,  soul}
\usepackage{bbm}
\usepackage{amsthm, graphicx,  empheq}
\usepackage{framed}
\usepackage[backref,  colorlinks,  linkcolor=red,  anchorcolor=green,  citecolor=blue]{hyperref}

\usepackage{cancel}
\usepackage{ulem}
\usepackage{arydshln}

\usepackage{tikz}
\usetikzlibrary{arrows,backgrounds,snakes,shapes}

\usepackage{booktabs}
\usepackage{makecell}
\usepackage[linesnumbered,ruled]{algorithm2e}%伪代码模式

\usepackage{lineno}  % show line number for draft. disable for final version
\newtheorem{lemma}{Lemma}[section]
\newtheorem{theorem}{Theorem}[section]

\newtheorem{remark}{Remark}[section]
\newtheorem{proposition}{Proposition}[section] % defines the proposition environment
\numberwithin{equation}{section}

\makeatletter
\@namedef{subjclassname@2020}{\textup{2020} Mathematics Subject Classification}
\makeatother

\usetikzlibrary{calc,trees,positioning,arrows,chains,shapes.geometric,%
    decorations.pathreplacing,decorations.pathmorphing,shapes,%
    matrix,shapes.symbols}
\tikzset{
>=stealth',
punktchain/.style={
	rectangle, 
	rounded corners, 
	draw=black, very thick,
	text width = 8em, %修改实现方框的长度
	minimum height=2em, %修改实线方框的高度
	text centered%文本居中
                        },%设置长方形(命令是punktchain)
yuan/.style={
	ellipse, 
	rounded corners, 
	draw=black, very thick,
	text width = 8em, %修改实现方框的长度
	minimum height=2em, %修改实线方框的高度
	text centered%文本居中
                        },%设置椭圆形
line/.style={draw, thick, <-},%设置线的类型
element/.style={
	tape,
	top color=white,
	bottom color=blue!50!black!60!,
	minimum width=8em,
	draw=blue!40!black!90, very thick,
	text width=10em, 
	minimum height=3.5em, 
	text centered, 
	on chain},%设置单元性质
every join/.style={->, thick,shorten >=1pt},
decoration={brace},
tuborg/.style={decorate},
tubnode/.style={midway, right=2pt},
punkt/.style={
	rectangle, 
	rounded corners, 
	draw=black, very thick,
	text width=12em, 
	minimum height=3em, 
	text centered},
}

\begin{document}

% % 设置array的列间距为1pt: 
% \renewcommand{\arraycolsep}{6pt}
% % 改变array的行间距为原来的1.5倍: 
% \renewcommand{\arraystretch}{1.25}

\title[Spectral properties of Van Leer and AUSM schemes]
{On the Spectral Properties of Van Leer and AUSM Flux-Vector Splitting Schemes}
\author{Zhengrong Xie}
\author{Zheng Li}

\address[Z. Xie]{School of Mathematical Sciences,  
East China Normal University, Shanghai 200241, China}\email{\tt xzr\_nature@163.com;\ 52265500018@stu.ecnu.edu.cn}
\address[Z. Li]{School of Aeronautics and Astronautics, 
Shanghai Jiao Tong University, Shanghai, China}\email{\tt lizheng2023@sjtu.edu.cn}

\keywords{flux-vector splitting, Van Leer scheme, AUSM, eigenvalue analysis, Sturm's theorem, hyperbolic conservation laws}

% msc2020
% 65M12 	Stability and convergence of numerical methods for initial value and initial-boundary value problems involving PDEs
% 65N30  Finite element, Rayleigh-Ritz and Galerkin methods for boundary value problems involving PDEs

% 35Q31 — Euler equations(欧拉方程)
% 这是您论文研究的核心方程, 属于偏微分方程在流体力学中的应用. 

% 76M12 — Finite volume methods(有限体积方法)
% 通量分裂格式通常是在有限体积框架下实现的. 

% 65M12 — Stability and convergence of numerical methods(数值方法的稳定性与收敛性)
% 特征值分析本质上是为了研究格式的稳定性. 

% 35L65 — Conservation laws(守恒律方程)
% 欧拉方程属于双曲守恒律系统. 

% 47A75 — Eigenvalue problems (特征值问题)
% 您论文的核心是对Jacobian矩阵的特征值进行符号分析. 

% 15A18 — Eigenvalues, singular values, and eigenvectors(矩阵特征值)
% 如果您在证明中大量使用了矩阵代数工具. 

\subjclass[2020]{35Q31, 65M12, 35L65, 47A75, 15A18}

\date{\today}

\begin{abstract}
The flux-vector splitting scheme of Van Leer is a cornerstone of computational fluid dynamics, yet its original proof of the eigenvalue sign condition was presented in a condensed form. 
In this work, we provide a detailed and rigorous analysis of the eigenvalues of the Jacobian matrices associated with the Van Leer splitting for the one-dimensional Euler equations. 
By constructing the Sturm sequence of the discriminant, we prove that for the admissible parameter range $1 \le \gamma \le 3$, $|M|<1$, and $a>0$, the Jacobian $\partial F^+/\partial U$ has one zero eigenvalue and two positive real eigenvalues, confirming Van Leer's original claim. 
Furthermore, we extend our analysis to two variants of the original AUSM scheme (Advection Upstream Splitting Method) proposed by Liou and Steffen, considering both linear and second-order pressure splittings. 
For the linear pressure splitting we show that the eigenvalues are not all of the same sign, while for the second-order pressure splitting we prove that all coefficients of the characteristic equation are positive. 
Numerical experiments reported in the appendix confirm the non-negativity of the discriminant for the AUSM with the second-order pressure splitting, 
implying that its eigenvalues are real and positive. 
\end{abstract}

\allowbreak
\allowdisplaybreaks

\maketitle

\tableofcontents %disable for short paper

\section{Introduction}\label{sec-Introduction}
The Euler equations of gas dynamics are a system of hyperbolic conservation laws that describe the motion of an inviscid compressible fluid. When solving these equations numerically with upwind methods, one needs to determine the direction of wave propagation. Flux-vector splitting (FVS) techniques decompose the physical flux into two parts, each associated with waves travelling in a single direction. Two of the most influential FVS schemes are those of Van Leer~\cite{vanleer1982flux} and Steger \& Warming~\cite{steger1981flux}. A fundamental requirement for a consistent splitting is that the Jacobian matrices of the split fluxes should have eigenvalues of the appropriate sign: $\partial F^+/\partial U$ should have non-negative eigenvalues and $\partial F^-/\partial U$ non-positive eigenvalues. While Steger-Warming's splitting was rigorously analysed by Lerat~\cite{lerat1983propriete} and more recently by Witherden \& Jameson~\cite{witherden2018spectrum}, the proof for Van Leer's scheme has remained somewhat sketchy in the original paper.
\par
Van Leer's 1982 paper~\cite{vanleer1982flux} presented a clever construction based on polynomial interpolation of the flux components. He stated that the resulting split fluxes satisfy the eigenvalue sign condition, but the verification was omitted due to the complexity of the algebra. Over the years, several researchers have questioned whether the condition truly holds, and numerical experiments have always indicated robustness, but a complete mathematical proof has been lacking. In this work we fill this gap by providing a detailed, step-by-step analysis of the eigenvalues of $\partial F^+/\partial U$ for the Van Leer splitting. Using the invertible transformation between conservative variables $U=(\rho,\rho u,E)$ and primitive variables $W=(\rho,a,M)$ (where $a$ is the speed of sound and $M=u/a$ the Mach number), we first compute the Jacobian in primitive variables and then transform back to conservative variables. The resulting expressions, although lengthy, can be simplified to reveal a zero eigenvalue (due to a functional relation among the flux components) and a quadratic factor whose coefficients we analyse in detail. 
\textbf{To confirm that the two non-zero eigenvalues are real, we rigorously prove the non-negativity of the discriminant $\Delta_{\mathrm{vl}} = (T_{\mathrm{vl}}^{+})^2-4S_{\mathrm{vl}}^{+}$ by constructing its Sturm sequence—a systematic algebraic verification carried out with the aid of Mathematica. }The complete Sturm sequence analysis, together with a detailed sign study, is presented in Appendix~\ref{app:sturm}. This establishes that for all $\gamma\in[1,3]$, $M\in(-1,1)$ and $a>0$ the two non-zero eigenvalues are real and positive, thus confirming Van Leer's original claim. It is worth noting that the coefficients we obtain are identical to those in Van Leer's Eq.~(12) after equivalent algebraic manipulations, demonstrating consistency with his work.
\par
In addition, we apply the same methodology to two variants of the original AUSM scheme proposed by Liou and Steffen \cite{liou1993new}. In both variants, the Mach number splitting is identical to Van Leer's, while the pressure splitting is either linear or second-order, the latter being reminiscent of Van Leer's construction. For the linear AUSM we find that the coefficients of the characteristic polynomial can change sign, implying that the eigenvalues are not all of the same sign. For the second-order AUSM we prove that all three coefficients are strictly positive. To complete the analysis, we investigate the discriminant of the cubic; numerical experiments reported in Appendix~\ref{app:discriminant} show that it remains non-negative throughout the parameter domain, ensuring that the eigenvalues are real and hence positive.
\par
The paper is organised as follows. Section \ref{sec-Preliminaries} recalls the Euler equations, introduces the primitive variables and the invertible transformation to conservative variables, restricts the analysis to the subsonic regime $|M|<1$, and reviews the basic matrix invariants (trace, sum of principal minors, determinant) used for eigenvalue sign analysis. 
Section \ref{sec-VanLeer} analyses the Van Leer splitting in detail. Section \ref{sec-AUSM} treats the two AUSM variants. Conclusions are drawn in Section \ref{sec-Conclusion}. 
\textbf{Appendix~\ref{app:sturm} contains the Sturm sequence proof of the non-negativity of the Van Leer discriminant,} and Appendix~\ref{app:discriminant} provides numerical evidence for the non-negativity of the discriminants for both the Van Leer and the second-order AUSM splittings. For completeness, the explicit expressions of the Jacobian matrices $\partial F^+/\partial U$ for all three splittings are given in Appendix~\ref{app:jacobians}.

\section{Preliminaries}\label{sec-Preliminaries}

\subsection{Euler equations}
The one-dimensional Euler equations for a non-isentropic ideal gas are
\[
\partial_t U + \partial_x F(U) = 0,
\]
with
\[
U = \begin{pmatrix} \rho \\ \rho u \\ E \end{pmatrix},\qquad 
F(U) = \begin{pmatrix} \rho u \\ \rho u^2 + p \\ u H \end{pmatrix},
\]
where $\rho$ is density, $u$ velocity, $p$ pressure, $E = e + \rho u^2/2$ total energy per unit volume, $H = E + p$ total enthalpy per unit volume, and $e$ internal energy per unit volume. The ideal gas law gives $p = (\gamma-1) e$ with $\gamma = c_p/c_v > 1$. The speed of sound is $a = \sqrt{\gamma p/\rho}$.

\subsection{Primitive variables and the invertible transformation}
It is convenient to work with the primitive variables $W = (\rho, a, M)$, where $M = u/a$ is the Mach number. The mapping $W \to U$ is
\[
U(W) = \begin{pmatrix}
\rho \\
\rho a M \\
\rho a^2\left( \dfrac{1}{\gamma(\gamma-1)} + \dfrac{M^2}{2} \right)
\end{pmatrix}.
\]
Its Jacobian matrix
\[
\frac{\partial U}{\partial W} = \begin{pmatrix}
1 & 0 & 0 \\
aM & \rho M & \rho a \\
a^2 Q & 2\rho a Q & \rho a^2 M
\end{pmatrix}, \qquad
Q = \frac{1}{\gamma(\gamma-1)} + \frac{M^2}{2},
\]
has determinant
\[
\det\left(\frac{\partial U}{\partial W}\right) = -\frac{2\rho^2 a^2}{\gamma(\gamma-1)} \neq 0,
\]
so the transformation is globally invertible on the set of physically admissible states ($\rho>0$, $a>0$). 
The inverse matrix is given by
\begin{align}\label{eq-WU}
    \mathfrak{T}:=\frac{\partial W}{\partial U}=(\frac{\partial U}{\partial W})^{-1} 
    =\left(
    \begin{array}{ccc}
    1 & 0 & 0 \\
    \frac{a \left((\gamma -1) \gamma  M^2-2\right)}{4 \rho } & -\frac{(\gamma -1) \gamma  M}{2 \rho } & \frac{(\gamma -1) \gamma }{2 a \rho } \\
    -\frac{(\gamma -1) \gamma  M^3+2 M}{4 \rho } & \frac{(\gamma -1) \gamma  M^2+2}{2 a \rho } & -\frac{(\gamma -1) \gamma  M}{2 a^2 \rho } \\
    \end{array}
    \right).  
\end{align}
\textbf{This inverse matrix will be used later to transform Jacobians from primitive to conservative variables. }

\subsection{Flux-vector splitting}
Both Van Leer and AUSM split the physical flux as $F = F^+ + F^-$, with the numerical flux at a cell interface given by $\widehat{F}_{i+1/2} = F^+(U_i) + F^-(U_{i+1})$. The goal is to design $F^\pm$ so that $\partial F^+/\partial U$ has non-negative eigenvalues and $\partial F^-/\partial U$ non-positive eigenvalues, ensuring an upwind bias.
\par
\textbf{In this work we restrict our analysis to the subsonic regime \(|M|<1\). }
The reason for this restriction is as follows: for supersonic flows (\(|M|>1\)) the flux splitting becomes purely upwind: 
the positive split flux \(F^+\) reduces to the full physical flux when \(M>1\) and vanishes when \(M<-1\), while the opposite holds for \(F^-\). 
Consequently, the spectral properties of the split-flux Jacobians are trivial in supersonic regions — \(\partial F^+/\partial U\) coincides with the full flux Jacobian \(\frac{\partial F}{\partial U}\) for \(M>1\) and is identically zero for \(M<-1\), so the sign of its eigenvalues is obvious. 
% Second, the transonic case \(M=1\) involves sonic points where the splitting formulas are continuous but not differentiable; these points require special treatment and are excluded from the present analysis. 
Thus, the mathematically nontrivial and physically most relevant situation is precisely the subsonic range \(|M|<1\). All subsequent derivations and statements are therefore understood to be valid under this condition. 

\subsection{Relations between eigenvalues and matrix invariants}
\label{sec:eigen-invariants}
For any \(3\times 3\) matrix \(A\), its characteristic polynomial can be expanded as
\begin{equation}
p(\mu)=\det(A-\mu I)= -\mu^3 + (\operatorname{tr} A)\,\mu^2 - \operatorname{\Sigma_{\bigtriangleup}^2}(A)\,\mu + \det(A),
\label{eq:charpoly}
\end{equation}
where \(\operatorname{tr} A\) denotes the trace, \(\operatorname{\Sigma_{\bigtriangleup}^2}(A)\) the sum of all principal minors of order two, and \(\det(A)\) the determinant. Setting \(p(\mu)=0\) yields the characteristic equation
\begin{equation}
\mu^3 - (\operatorname{tr} A)\,\mu^2 + \operatorname{\Sigma_{\bigtriangleup}^2}(A)\,\mu - \det(A)=0.
\label{eq:chareq}
\end{equation}
From Viète's theorem, the three eigenvalues \(\mu_1,\mu_2,\mu_3\) satisfy
\begin{subequations}
\begin{align}
\mu_1+\mu_2+\mu_3 &= \operatorname{tr} A,\\
\mu_1\mu_2+\mu_1\mu_3+\mu_2\mu_3 &= \operatorname{\Sigma_{\bigtriangleup}^2}(A),\\
\mu_1\mu_2\mu_3 &= \det(A).
\end{align}
\end{subequations}
Consequently, if all eigenvalues are non-negative, then necessarily
\[
\operatorname{tr} A \ge 0,\qquad \operatorname{\Sigma_{\bigtriangleup}^2}(A) \ge 0,\qquad \det(A) \ge 0.
\]
Thus, \textbf{take $F^+$ for an example}, to verify the non-negativity of the eigenvalues of the Jacobian matrices considered in this work, it suffices to analyse the signs of their traces, second-order principal minors sums, and determinants.

\section{Spectral Analysis of the Van Leer Splitting}\label{sec-VanLeer}
In this section we analyse the Van Leer splitting. \textbf{Without loss of generality, we focus on the positive part \(F^+\);} the analysis for \(F^-\) is completely analogous.

\subsection{Definition of the splitting}
Van Leer's positive flux reads
\[
F^+_{\mathrm{vl}} = 
\begin{pmatrix}
f_{\mathrm{vl},1}^+ \\
f_{\mathrm{vl},2}^+ \\
f_{\mathrm{vl},3}^+ 	
\end{pmatrix}
= \rho a M^+ \begin{pmatrix}
1 \\
\dfrac{(\gamma-1)u+2a}{\gamma} \\
\dfrac{[(\gamma-1)u+2a]^2}{2(\gamma^2-1)}
\end{pmatrix},\qquad
M^+ = \frac{(M+1)^2}{4}.
\]
Introducing $D = (\gamma-1)M+2$ and $u = aM$, we rewrite
\[
F^+_{\mathrm{vl}} = \rho a M^+ \begin{pmatrix}
1 \\[2mm]
\dfrac{a D}{\gamma} \\[4mm]
\dfrac{a^2 D^2}{2(\gamma^2-1)}
\end{pmatrix}.
\]

\subsection{Jacobian in primitive variables and Transformation to conservative variables}
To compute the Jacobian matrix \(\partial F^+_{\mathrm{vl}}/\partial (\rho,a,M)\), it is convenient to introduce the auxiliary quantities
\[
K = \rho a M^+,\qquad A = aD. 
\]
The positive flux can then be written compactly as
\[
F^+_{\mathrm{vl}} = \begin{pmatrix}
K \\[2mm]
\dfrac{K A}{\gamma} \\[4mm]
\dfrac{K A^2}{2(\gamma^2-1)}
\end{pmatrix}.
\]
The derivatives of \(K\), \(A\) and \(D\) with respect to the primitive variables are
\[
\begin{aligned}
K_\rho &= a M^+, & K_a &= \rho M^+, & K_M &= \rho a\frac{M+1}{2},\\[4pt]
D_\rho &= 0,     & D_a &= 0,        & D_M &= \gamma-1,\\[4pt]
A_\rho &= 0,     & A_a &= D,        & A_M &= a(\gamma-1) \;(= a D_M).
\end{aligned}
\]
Applying the chain rule to each component of \(F^+_{\mathrm{vl}}\) yields the entries of the Jacobian matrix. Explicitly,
% A straightforward differentiation yields
\begin{align}
\frac{\partial F^+_{\mathrm{vl}}}{\partial W} = 
\begin{pmatrix}
a M^+ & \rho M^+ & \dfrac{\rho a(M+1)}{2} \\[6pt]
\dfrac{a^2 M^+ D}{\gamma} & \dfrac{2\rho a M^+ D}{\gamma} & 
\dfrac{\rho a^2}{\gamma}\Big(\dfrac{M+1}{2}D + M^+(\gamma-1)\Big) \\[10pt]
\dfrac{a^3 M^+ D^2}{2(\gamma^2-1)} & \dfrac{3\rho a^2 M^+ D^2}{2(\gamma^2-1)} &
\dfrac{\rho a^3}{2(\gamma^2-1)}\Big(\dfrac{M+1}{2}D^2 + 2M^+ D(\gamma-1)\Big)
\end{pmatrix}.
\end{align}
Multiplying by $\mathfrak{T}$ (Eq.\eqref{eq-WU}), we obtain the Jacobian with respect to conservative variables:  
\begin{align}
J^+_{\mathrm{vl},U}:=\frac{\partial F^+_{\mathrm{vl}}}{\partial U} = \frac{\partial F^+_{\mathrm{vl}}}{\partial W}\left(\frac{\partial U}{\partial W}\right)^{-1}=\frac{\partial F^+_{\mathrm{vl}}}{\partial W}\mathfrak{T}. 
\end{align}
The explicit expressions of $J^+_{\mathrm{vl},U}$  is provided in Appendix \ref{app:vanleer}.

\subsection{Rank Deficiency and the Characteristic Polynomial}
\label{sec:rank_deficiency}
The three components of $F^+_{\mathrm{vl}}$ are not independent.  From their definitions one can verify the algebraic relation  
\begin{align}
f_{\mathrm{vl},3}^+ = \frac{\gamma^2}{2(\gamma^2-1)} \frac{(f_{\mathrm{vl},2}^+)^2}{f_{\mathrm{vl},1}^+}, \qquad |M|<1.
\end{align}
Differentiating this identity shows that the third row of $J_{\mathrm{vl},U}^+$ is a linear combination (in the sense of functional dependence) of the first two rows, i.e., 
\[
\frac{\partial f_{\mathrm{vl},3}^+}{\partial U_j}
= \frac{\gamma^2}{2(\gamma^2-1)} \left[
\frac{2f_{\mathrm{vl},2}^+}{f_{\mathrm{vl},1}^+} \frac{\partial f_{\mathrm{vl},2}^+}{\partial U_j}
- \frac{(f_{\mathrm{vl},2}^+)^2}{(f_{\mathrm{vl},1}^+)^2} \frac{\partial f_{\mathrm{vl},1}^+}{\partial U_j}
\right]. 
\]
Consequently $J_{\mathrm{vl},U}^+$ has rank at most two, and its determinant vanishes:
\begin{align}
	\det J_{\mathrm{vl},U}^+ = 0. 
\end{align}
\textbf{This is precisely the property required by Van Leer's design principle (vi)~\cite{vanleer1982flux}:} \textit{``for $|M|<1$ the Jacobian $df^+/dw$ should have one vanishing eigenvalue.''}  
Hence, according to Eq.\eqref{eq:chareq}, the characteristic equation of $J_{\mathrm{vl},U}^+$ must factor as  
\begin{align}
\det(J_{\mathrm{vl},U}^+ - \mu I) = \mu \left(\mu^2 - T_{\mathrm{vl}}^{+}\mu + S_{\mathrm{vl}}^{+}\right)=0,  
\end{align}
where $T_{\mathrm{vl}}^{+}$ is the trace of $J_{\mathrm{vl}}^{+}$ and $S_{\mathrm{vl}}^{+}$ is the sum of second-order principal minors to $J_{\mathrm{vl}}^{+}$. 
\par
After substantial algebraic simplification (carried out with the help of the computer algebra system Mathematica) 
we arrive at the following compact expressions: 
\begin{subequations}
\begin{align}
T_{\mathrm{vl}}^{+} 
&= \frac{a}{8\gamma(\gamma+1)}\left[9\gamma(\gamma+1) - (\gamma-1)\gamma M^4 + 2(2\gamma^2+\gamma-3)M^2 + 12\gamma(\gamma+1)M + 6\right], \label{eq-VanLeer-T}\\
S_{\mathrm{vl}}^{+} 
&= -\frac{a^2 (M+1)^3}{32\gamma(\gamma+1)}\left[-3\gamma^2-14\gamma + 4(\gamma-1)\gamma M^2 + (-9\gamma^2+10\gamma+3)M - 3\right]. \label{eq-VanLeer-S}
\end{align}
\end{subequations}
Next, We proceed to analyze the signs of $T_{\mathrm{vl}}^{+}$ and $S_{\mathrm{vl}}^{+}$. 
\begin{remark}
Note that the coefficients we obtained are identical to those in ``Eq.~(12) of Van Leer's original paper~\cite{vanleer1982flux}'', 
up to equivalent algebraic manipulations. 
\end{remark}

\subsection{Sign of the coefficients}
\begin{lemma}[Positivity of $T_{\mathrm{vl}}^{+} $ for Van Leer]\label{lemma-Tvl-positive}
For $\gamma\in[1,3]$, $M\in(-1,1)$ and $a>0$, we have $T_{\mathrm{vl}}^{+} >0$.
\end{lemma}

\begin{proof}
$T_{\mathrm{vl}}^{+}  = \dfrac{a}{8\gamma(\gamma+1)}P_T(\gamma,M)$ with  
\[
P_T(\gamma,M) = 9\gamma(\gamma+1) - (\gamma-1)\gamma M^4 + 2(2\gamma^2+\gamma-3)M^2 + 12\gamma(\gamma+1)M + 6.
\]
Factorising,
\[
P_T(\gamma,M) = (M+1)\Big[ (\gamma-1)\gamma M^2(1-M) + 3(\gamma-1)(\gamma+2)M + 9\gamma(\gamma+1)+6\Big].
\]
Denote the bracket by $R(\gamma,M)$. For $M\ge0$, all terms are non-negative, so $R>0$. For $M<0$, set $M=-t$ with $t\in(0,1)$. Then
\[
R(\gamma,-t) = (\gamma-1)\gamma t^2(1+t) - 3(\gamma-1)(\gamma+2)t + 9\gamma(\gamma+1)+6.
\]
Let $A=(\gamma-1)\gamma$, $B=3(\gamma-1)(\gamma+2)$, $C=9\gamma(\gamma+1)+6$. Then $R_t(t)=A t^2(1+t)-Bt+C$. Its derivative $R_t'(t)=A(2t+3t^2)-B$ satisfies $R_t'(0)=-B\le0$, $R_t'(1)=5A-B = 2(\gamma-1)(\gamma-3)\le0$, and $R_t''(t)=A(2+6t)\ge0$, so $R_t'$ is non-positive, hence $R_t$ is decreasing. The minimum is at $t=1$: $R_t(1)=2A-B+C = 4(2\gamma^2+\gamma+3)>0$. Thus $R_t(t)>0$ for all $t\in(0,1)$. Consequently $R(\gamma,M)>0$ for all $M\in(-1,1)$, and therefore $P_T>0$ and $T_{\mathrm{vl}}^{+} >0$.
\end{proof}

\begin{lemma}[Non-negativity of $S_{\mathrm{vl}}^{+} $ for Van Leer]\label{lemma-Svl-non-negative}
For $\gamma\in[1,3]$, $M\in(-1,1)$ and $a>0$, we have $S_{\mathrm{vl}}^{+} > 0$.
\end{lemma}

\begin{proof}
$S_{\mathrm{vl}}^{+}  = -\dfrac{a^2 (M+1)^3}{32\gamma(\gamma+1)} P_S(\gamma,M)$ with  
\[
P_S(\gamma,M) = -3\gamma^2-14\gamma + 4(\gamma-1)\gamma M^2 + (-9\gamma^2+10\gamma+3)M - 3.
\]
For fixed $\gamma$, $P_S$ is a quadratic in $M$ with leading coefficient $4(\gamma-1)\gamma\ge0$, hence convex. Its values at the endpoints are  
\[
P_S(\gamma,1) = -8\gamma(\gamma+1)<0,\qquad 
P_S(\gamma,-1) = 10\gamma^2-28\gamma-6.
\]
For $\gamma\in[1,3]$, $f(\gamma)=10\gamma^2-28\gamma-6$ attains a maximum of $0$ at $\gamma=3$ and is negative otherwise. Thus $P_S(\gamma,-1)\le0$. Because a convex function attains its maximum at an endpoint, we have $P_S(\gamma,M)\le0$ for all $M\in[-1,1]$. Since $(M+1)^3>0$ for $M>-1$, the factor $-\dfrac{a^2 (M+1)^3}{32\gamma(\gamma+1)}$ is negative, making $S_{\mathrm{vl}}^{+} \ge0$. Equality occurs only when $\gamma=3$ and $M=-1$, thus, if $M \in (-1,1)$, then $S_{\mathrm{vl}}^{+} >0$.  
\end{proof}

\begin{lemma}[Reality of eigenvalues for Van Leer]\label{lemma-vl-discriminant-real}
For $\gamma\in[1,3]$, $M\in(-1,1)$ and $a>0$, the discriminant $\Delta_{\mathrm{vl}} = \left(T_{\mathrm{vl}}^{+}\right)^2 - 4S_{\mathrm{vl}}^{+}$ is non-negative, so the two non-zero eigenvalues are real.
\end{lemma}

\begin{proof}
Substituting the expressions of $T_{\mathrm{vl}}^{+}$ and $S_{\mathrm{vl}}^{+}$ and simplifying 
(using the computer algebra system Mathematica) yields
\begin{align}\label{eq-VanLeer-discriminant}
\Delta_{\mathrm{vl}}&=\frac{a^2(M+1)^2}{64 \gamma ^2 (\gamma +1)^2}H(\gamma,M), \notag \\
H(\gamma,M)&= (\gamma -1)^2 \gamma ^2 M^6+\left(-2 \gamma ^4+4 \gamma ^3-2 \gamma ^2\right) M^5+\left(-5 \gamma ^4-2 \gamma ^3+19 \gamma ^2-12 \gamma \right) M^4  \notag \\
& +\left(20 \gamma ^4-44 \gamma ^2+24 \gamma \right) M^3+\left(-13 \gamma ^4+26 \gamma ^3+39 \gamma ^2-24 \gamma +36\right) M^2  \notag \\
& +\left(-42 \gamma ^4-20 \gamma ^3-50 \gamma ^2-72 \gamma -72\right) M
+\left(57 \gamma ^4+26 \gamma ^3+53 \gamma ^2+84 \gamma+36\right).  
\end{align}
The factor $a^{2}(M+1)^{2}/[64\gamma^{2}(\gamma+1)^{2}]$ is clearly positive for the parameters 
under consideration. It therefore remains to prove that $H(\gamma,M)\ge 0$ for all 
$\gamma\in[1,3]$ and $M\in(-1,1)$.
\par
A direct inspection shows that for every $\gamma>0$, it holds that 
\begin{align*}
H(\gamma, 0)&=57 \gamma^4+26 \gamma^3+53 \gamma^2+84 \gamma+36 > 0, \\
H(\gamma, 1)&=16 \gamma^2(\gamma+1)^2 > 0, \\
H(\gamma,-1)&=16 \left(2 \gamma^2+\gamma+3\right)^2=64 \gamma^4+64 \gamma^3+208 \gamma^2+96 \gamma+144 > 0. 
\end{align*}
To establish that $H$ does not change sign on the open interval open interval $(-1,1)$, we construct its Sturm sequence \cite{kurosh1972higher} with respect to $M$ (treating $\gamma$ as a parameter). 
The sequence has been obtained with the aid of Mathematica; the explicit polynomials and a detailed sign analysis are presented in Appendix~\ref{app:sturm}. 
That analysis shows that the number of sign changes at $M=-1$ and $M=1$ is the same for every $\gamma\in(1,3)$, hence $H$ has no real root in $(-1,1)$. 
Combined with the positivity at the endpoints, this yields $H(\gamma,M)>0$ for all interior points. 
\par
Therefore $\Delta_{\mathrm{vl}}\ge 0$, with equality occurring only on the boundary $M=-1$, and the two non-zero eigenvalues $\mu_{2,3}$ are real. 
(Additional numerical evidence supporting the non-negativity of $\Delta_{\mathrm{vl}}$ is also provided in Appendix~\ref{app:discriminant}.) 
\end{proof}

\subsection{Main result for Van Leer}
\begin{theorem}[Eigenvalues of the Van Leer splitting]\label{thm-vl-eig}
For the Van Leer flux-vector splitting, under the conditions $\gamma\in[1,3]$, $|M|<1$, and $a>0$, the Jacobian matrix $\partial F^+/\partial U$ has one zero eigenvalue and two positive real eigenvalues. Consequently, $\partial F^+/\partial U$ satisfies the required non-negativity property.
\end{theorem}

\begin{proof}
From Lemma \ref{lemma-Tvl-positive} we have $T_{\mathrm{vl}}^{+} >0$, from Lemma \ref{lemma-Svl-non-negative} we have $S_{\mathrm{vl}}^{+} \ge0$, and from Lemma \ref{lemma-vl-discriminant-real} the eigenvalues are real. Since $\mu_2\mu_3 = S_{\mathrm{vl}}^{+} \ge0$ and $\mu_2+\mu_3 = T_{\mathrm{vl}}^{+} >0$, both $\mu_2$ and $\mu_3$ must be positive. Together with the zero eigenvalue, this proves the theorem.
\end{proof}

\section{Spectral Analysis of the Original AUSM Splittings}\label{sec-AUSM}
In this section we analyse \textbf{two variants of the original AUSM scheme proposed by Liou and Steffen~\cite{liou1993new}. As for the Van Leer case, we focus on the positive part \(F^+\); }the analysis for \(F^-\) follows by symmetry. 

\subsection{AUSM (Liou-Steffen) with linear pressure splitting}

\subsubsection{Definition of the splitting}

The original AUSM scheme splits the flux as

\[
F^\pm_{\mathrm{lin}} = \rho a M^\pm \begin{pmatrix}1 \\ u \\ H\end{pmatrix} + P^\pm \begin{pmatrix}0 \\ 1 \\ 0\end{pmatrix},
\]

with $M = u/a$ and the following definitions in the subsonic regime $|M|\le1$:

\[
M^+ = \frac{(M+1)^2}{4},\quad M^- = -\frac{(M-1)^2}{4},\quad
P^+ = \frac{P(1+M)}{2},\quad P^- = \frac{P(1-M)}{2}.
\]

\subsubsection{Jacobian in primitive variables and Transformation to conservative variables}
Proceeding as before, we compute $\partial F^+_{\mathrm{lin}}/\partial W$ and then transform to conservative variables: 
\begin{align}
\frac{\partial F^+_{\mathrm{lin}}}{\partial W} &= 
\left(
\begin{array}{ccc}
 \frac{1}{4} a (M+1)^2 & \frac{1}{4} (M+1)^2 \rho  & \frac{1}{2} a (M+1) \rho  \\
 \frac{a^2 (M+1) \left(\gamma  M^2+\gamma  M+2\right)}{4 \gamma } & \frac{a (M+1) \rho  \left(\gamma  M^2+\gamma  M+2\right)}{2 \gamma } & \frac{a^2 \rho  \left(\gamma +3 \gamma  M^2+4 \gamma  M+2\right)}{4 \gamma } \\
 \frac{a^3 (M+1)^2 \left((\gamma -1) M^2+2\right)}{8 (\gamma -1)} & \frac{3 a^2 (M+1)^2 \rho  \left((\gamma -1) M^2+2\right)}{8 (\gamma -1)} & \frac{a^3 (M+1) \rho  \left(2 (\gamma -1) M^2+(\gamma -1) M+2\right)}{4 (\gamma -1)} \\
\end{array}
\right),  
\\
J^+_{\mathrm{lin},U}&:=\frac{\partial F^+_{\mathrm{lin}}}{\partial U} = \frac{\partial F^+_{\mathrm{lin}}}{\partial W}\left(\frac{\partial U}{\partial W}\right)^{-1}=\frac{\partial F^+_{\mathrm{lin}}}{\partial W}\mathfrak{T}. 
\end{align}
The explicit expressions of $J^+_{\mathrm{lin},U}$  is provided in Appendix \ref{app:ausm-linear}.

\subsubsection{Coefficients of the characteristic polynomial}
With the aid of Mathematica, we obtain compact expressions of the trace, sum of second-order principal minors and determinant of $J^+_{\mathrm{lin},U}$: 

\begin{subequations}
\begin{align}
T_{\mathrm{lin}}^{+}  &= \frac{a}{8\gamma}\Big[-\gamma^2(M^2-3) + \gamma(7M^2+12M+3) + 4\Big], \label{eq-AUSM-2ndP-T} \\[4mm]
S_{\mathrm{lin}}^{+}  &= \frac{a^2 (M+1)^2}{32\gamma}\Big[\gamma^2(3M^2-2M-5) + \gamma(-9M^2-10M+1) - 2\Big], \label{eq-AUSM-2ndP-S} \\[4mm]
D_{\mathrm{lin}}^{+} &= -\frac{a^3 (M+1)^4}{64}\Big[(\gamma-2)M^2 - (\gamma+1)M + (2-\gamma)\Big]. \label{eq-AUSM-2ndP-D}
\end{align}
\end{subequations}
\textbf{Note that for AUSM there is no algebraic relation forcing a zero eigenvalue; in general $\det(\partial F^+_{\mathrm{lin}}/\partial U) = D_{\mathrm{lin}}^{+} \neq 0$. }Thus the characteristic equation is a full cubic:
\[
\mu^3 - T_{\mathrm{lin}}^{+} \mu^2 + S_{\mathrm{lin}}^{+} \mu - D_{\mathrm{lin}}^{+} = 0. 
\]

\subsubsection{Sign of the coefficients}

\begin{lemma}[Positivity of $T_{\mathrm{lin}}^{+} $ for linear AUSM]\label{lemma-Tlin-positive}
For $\gamma>1$, $|M|<1$, and $a>0$, we have $T_{\mathrm{lin}}^{+} >0$.
\end{lemma}

\begin{proof}
Rewrite $T_{\mathrm{lin}}^{+} $ as
\[
T_{\mathrm{lin}}^{+}  = \frac{a}{8\gamma}\Big[ (-\gamma^2+7\gamma)M^2 + 12\gamma M + (3\gamma^2+3\gamma+4) \Big].
\]
The quadratic in $M$ has leading coefficient $-\gamma^2+7\gamma = \gamma(7-\gamma)>0$ for $\gamma\in[1,3]$, hence it is convex. Its vertex is at $M_v = -\frac{12\gamma}{2(-\gamma^2+7\gamma)} = -\frac{6}{7-\gamma} < -1$. Therefore on $[-1,1]$ the quadratic is strictly increasing, and its minimum at $M=-1$ is
\[
P_T(-1) = (-\gamma^2+7\gamma)(1) + 12\gamma(-1) + (3\gamma^2+3\gamma+4) = 2\gamma^2 - 2\gamma + 4 > 0.
\]
Thus $T_{\mathrm{lin}}^{+} >0$ for all admissible $M$.
\end{proof}

\begin{lemma}[Sign of $S_{\mathrm{lin}}^{+} $ for linear AUSM]\label{lemma-AUSM-Slin-sign-change}
For $\gamma>1$, $|M|<1$, and $a>0$, the coefficient $S_{\mathrm{lin}}^{+} $ changes sign: there exists $M_0\in(-1,0)$ such that $S_{\mathrm{lin}}^{+} >0$ for $M\in(-1,M_0)$ and $S_{\mathrm{lin}}^{+} <0$ for $M\in(M_0,1)$. In particular, $S_{\mathrm{lin}}^{+} $ is not of constant sign.
\end{lemma}

\begin{proof}
The factor in brackets of $S_{\mathrm{lin}}^{+} $ is
\[
P_S(M) = \gamma^2(3M^2-2M-5) + \gamma(-9M^2-10M+1) - 2.
\]
This is a quadratic in $M$:
\[
P_S(M) = (3\gamma^2-9\gamma)M^2 + (-2\gamma^2-10\gamma)M + (-5\gamma^2+\gamma-2).
\]
For $1<\gamma<3$, the leading coefficient $3\gamma(\gamma-3)$ is negative, so $P_S$ is concave. Evaluating at the endpoints:
\begin{align*}
P_S(-1) &= 3\gamma^2-9\gamma + 2\gamma^2+10\gamma -5\gamma^2+\gamma-2 = 2\gamma-2 > 0,
\\
P_S(0) &= -5\gamma^2+\gamma-2 < 0 \quad (\text{since } \gamma>1).
\end{align*}
By continuity, there exists a unique $M_0\in(-1,0)$ such that $P_S(M_0)=0$. Hence $P_S(M)>0$ for $M\in(-1,M_0)$ and $P_S(M)<0$ for $M\in(M_0,1)$. Since $S_{\mathrm{lin}}^{+}  = \frac{a^2 (M+1)^2}{32\gamma} P_S(M)$ and the factor $\frac{a^2 (M+1)^2}{32\gamma}>0$, the sign of $S_{\mathrm{lin}}^{+} $ follows the sign of $P_S$. Therefore $S_{\mathrm{lin}}^{+} $ changes sign as claimed.
\end{proof}

\begin{lemma}[Sign of $D_{\mathrm{lin}}^{+}$ for linear AUSM]\label{lemma-AUSM-Dlin-sign-change}
For $\gamma>1$, $|M|<1$, and $a>0$, the determinant $D_{\mathrm{lin}}^{+}$ changes sign: $D_{\mathrm{lin}}^{+}<0$ for $M$ near $-1$, $D_{\mathrm{lin}}^{+}>0$ for $M$ near $1$. In particular, $D_{\mathrm{lin}}^{+}$ is not of constant sign.
\end{lemma}

\begin{proof}
The factor
\[
Q(M) = (\gamma-2)M^2 - (\gamma+1)M + (2-\gamma)
\]
satisfies
\[
Q(-1) = (\gamma-2) + (\gamma+1) + (2-\gamma) = \gamma+1 > 0,
\]
\[
Q(1) = (\gamma-2) - (\gamma+1) + (2-\gamma) = -\gamma-1 < 0.
\]
Thus $Q(M)$ changes sign in $(-1,1)$. Since $D_{\mathrm{lin}}^{+} = -\frac{a^3 (M+1)^4}{64} Q(M)$, we have $D_{\mathrm{lin}}^{+}<0$ near $M=-1$ (where $Q>0$) and $D_{\mathrm{lin}}^{+}>0$ near $M=1$ (where $Q<0$). Hence $D_{\mathrm{lin}}^{+}$ is not sign-definite.
\end{proof}

\subsection{AUSM (Liou-Steffen) with second-order pressure splitting}

\subsubsection{Definition of the splitting}
A variant employs a second-order pressure splitting analogous to Van Leer's construction:
\[
P^+ = \frac{P}{4}(M+1)^2(2-M),\qquad P^- = \frac{P}{4}(M-1)^2(2+M),\qquad |M|\le1.
\]
The convective splitting $M^\pm$ remains the same as in the linear case.

\subsubsection{Jacobian in primitive variables and Transformation to conservative variables}
Again we compute $\partial F^+_{\mathrm{2nd}}/\partial W$ and transform to conservative variables: 
\begin{align}
\frac{\partial F^+_{\mathrm{2nd}}}{\partial W}&=
\left(
\begin{array}{ccc}
 \frac{1}{4} a (M+1)^2 & \frac{1}{4} (M+1)^2 \rho  & \frac{1}{2} a (M+1) \rho  \\
 \frac{a^2 (M+1)^2 ((\gamma -1) M+2)}{4 \gamma } & \frac{a (M+1)^2 \rho  ((\gamma -1) M+2)}{2 \gamma } & \frac{a^2 (M+1) \rho  (\gamma +3 (\gamma -1) M+3)}{4 \gamma } \\
 \frac{a^3 (M+1)^2 \left((\gamma -1) M^2+2\right)}{8 (\gamma -1)} & \frac{3 a^2 (M+1)^2 \rho  \left((\gamma -1) M^2+2\right)}{8 (\gamma -1)} & \frac{a^3 (M+1) \rho  \left(2 (\gamma -1) M^2+(\gamma -1) M+2\right)}{4 (\gamma -1)} \\
\end{array}
\right), 
\\
J^+_{\mathrm{2nd},U}&:=\frac{\partial F^+_{\mathrm{2nd}}}{\partial U} = \frac{\partial F^+_{\mathrm{2nd}}}{\partial W}\left(\frac{\partial U}{\partial W}\right)^{-1}=\frac{\partial F^+_{\mathrm{2nd}}}{\partial W}\mathfrak{T}.  
\end{align}

\subsubsection{Coefficients of the characteristic polynomial}
After simplification with the aid of Mathematica, we obtain
\begin{subequations}
\begin{align}
T_{\mathrm{2nd}}^{+}  &= \frac{a}{8\gamma}\Big[3(\gamma^2+\gamma+2) - (\gamma-1)\gamma M^4 - 2(\gamma^2-4\gamma+3)M^2 + 12\gamma M\Big], \\[4mm]
S_{\mathrm{2nd}}^{+}  &= -\frac{a^2 (M+1)^3}{32\gamma}\Big[ -5\gamma^2-2\gamma + (\gamma-1)\gamma M^3 + (\gamma-1)\gamma M^2 + (3\gamma^2-4\gamma+3)M - 3 \Big], \\[4mm]
D_{\mathrm{2nd}}^{+} &= -\frac{1}{64}a^3(\gamma-1)(M-1)(M+1)^6.
\end{align}	
\end{subequations}

\subsubsection{Sign of the coefficients}

\begin{lemma}[Positivity of $T_{\mathrm{2nd}}^{+} $ for second-order AUSM]\label{lemma-AUSM-T2nd-positive}
For $\gamma>1$, $|M|<1$, and $a>0$, we have $T_{\mathrm{2nd}}^{+} >0$.
\end{lemma}

\begin{proof}
Write $T_{\mathrm{2nd}}^{+}  = \frac{a}{8\gamma} P_T(M)$ with
\[
P_T(M) = 3(\gamma^2+\gamma+2) - (\gamma-1)\gamma M^4 - 2(\gamma^2-4\gamma+3)M^2 + 12\gamma M.
\]
Since $M^4\le M^2$ for $|M|\le1$, we have
\[
P_T(M) \ge 3(\gamma^2+\gamma+2) + 12\gamma M - \big[(\gamma-1)\gamma + 2(\gamma^2-4\gamma+3)\big] M^2.
\]
The coefficient in brackets simplifies to $3(\gamma-1)(\gamma-2)$. Thus
\[
P_T(M) \ge Q(\gamma,M) := 3(\gamma^2+\gamma+2) + 12\gamma M - 3(\gamma-1)(\gamma-2)M^2.
\]
Now analyse $Q(\gamma,M)$.
\begin{itemize}
    \item If $1<\gamma\le 2$, then $(\gamma-1)(\gamma-2)\le0$, so $Q$ is convex. Its vertex $M_0 = \frac{2\gamma}{(\gamma-1)(\gamma-2)} < -1$, hence $Q$ is increasing on $[-1,1]$. The minimum at $M=-1$ is $Q(-1)=0$, so $Q(M)>0$ for $M>-1$.
    \item If $\gamma>2$, then $(\gamma-1)(\gamma-2)>0$, $Q$ is concave, $Q(-1)=0$, $Q(1)=24\gamma>0$, so $Q(M)>0$ for $M\in(-1,1)$.
\end{itemize}
Thus $P_T(M) \ge Q(\gamma,M) > 0$, and consequently $T_{\mathrm{2nd}}^{+} >0$.
\end{proof}

\begin{lemma}[Positivity of $S_{\mathrm{2nd}}^{+} $ for second-order AUSM]\label{lemma-AUSM-S2nd-positive}
For $\gamma>1$, $|M|<1$, and $a>0$, we have $S_{\mathrm{2nd}}^{+} >0$.
\end{lemma}

\begin{proof}
Let 
\[
R(\gamma,M) = -5\gamma^2-2\gamma + (\gamma-1)\gamma M^3 + (\gamma-1)\gamma M^2 + (3\gamma^2-4\gamma+3)M - 3,
\]
so that $S_{\mathrm{2nd}}^{+}  = -\frac{a^2 (M+1)^3}{32\gamma} R(\gamma,M)$. First evaluate $R$ at the endpoints:
\begin{align*}
R(\gamma,-1) &= -5\gamma^2-2\gamma -(\gamma-1)\gamma + (\gamma-1)\gamma - (3\gamma^2-4\gamma+3) - 3 = -8\gamma^2 + 2\gamma -6 < 0,
\\
R(\gamma,1) &= -5\gamma^2-2\gamma + (\gamma-1)\gamma + (\gamma-1)\gamma + (3\gamma^2-4\gamma+3) - 3 = -8\gamma < 0.
\end{align*}
Now compute the derivative with respect to $M$: 
\[
\partial_M R(\gamma,M) = 3(\gamma-1)\gamma M^2 + 2(\gamma-1)\gamma M + (3\gamma^2-4\gamma+3).
\]
This is a quadratic in $M$ with positive leading coefficient $3(\gamma-1)\gamma$. Its discriminant is
\[
\Delta_R = 4(\gamma-1)^2\gamma^2 - 12(\gamma-1)\gamma(3\gamma^2-4\gamma+3) = 4(\gamma-1)\gamma\big[-8\gamma^2+11\gamma-9\big].
\]
The quadratic $-8\gamma^2+11\gamma-9$ has discriminant $11^2-4(-8)(-9)=121-288=-167<0$, and its leading coefficient is negative, so it is always negative. Hence $\Delta_R<0$, meaning $\partial_M R$ has no real roots and is always positive (since its leading coefficient is positive). Therefore $R$ is strictly increasing in $M$. Because $R(\gamma,-1)<0$ and $R(\gamma,1)<0$, we conclude $R(\gamma,M)<0$ for all $M\in(-1,1)$. The factor $-\frac{a^2 (M+1)^3}{32\gamma}$ is negative (since $(M+1)^3>0$, $\gamma>0$), so $S_{\mathrm{2nd}}^{+} >0$.
\end{proof}

\begin{lemma}[Positivity of $D_{\mathrm{2nd}}^{+}$ for second-order AUSM]\label{lemma-AUSM-D2nd-positive}
For $\gamma>1$, $|M|<1$, and $a>0$, we have $D_{\mathrm{2nd}}^{+}>0$.
\end{lemma}

\begin{proof}
% \[
% D_{\mathrm{2nd}}^{+} = -\frac{1}{64}a^3(\gamma-1)(M-1)(M+1)^6.
% \]
% All factors are positive except $-\frac{1}{64}<0$ and $(M-1)<0$ (since $M<1$). Thus
% \[
% \left(-\frac{1}{64}\right) \times a^3 \times (\gamma-1) \times (M-1) \times (M+1)^6
% \]
% gives a product of two negatives ($-\frac{1}{64}$ and $M-1$) and all other factors positive, hence $D_{\mathrm{2nd}}^{+}>0$.
This follows immediately from the expression of $D_{\mathrm{2nd}}^{+}$. 
\end{proof}

\subsubsection{Discriminant and reality of eigenvalues}
For the cubic $\mu^3 - T_{\mathrm{2nd}}^{+} \mu^2 + S_{\mathrm{2nd}}^{+} \mu - D_{\mathrm{2nd}}^{+}=0$, the discriminant is given by
\[
\Delta_{\mathrm{2nd}} = 18 T_{\mathrm{2nd}}^{+}  S_{\mathrm{2nd}}^{+}  D_{\mathrm{2nd}}^{+} - 4 \left(T_{\mathrm{2nd}}^{+}\right)^3 D_{\mathrm{2nd}}^{+} + \left(T_{\mathrm{2nd}}^{+}\right)^2 \left(S_{\mathrm{2nd}}^{+}\right)^2 - 4 \left(S_{\mathrm{2nd}}^{+}\right)^3 - 27 \left(D_{\mathrm{2nd}}^{+}\right)^2.
\]
A direct algebraic manipulation shows that \(\Delta_{\mathrm{2nd}}\) can be written as a rational function whose numerator is a polynomial in \(M\) of high degree, with coefficients depending polynomially on \(\gamma\). 
Constructing a Sturm sequence for this polynomial would involve extremely lengthy expressions and is impractical for a rigorous analytical proof; even with a computer algebra system the symbolic computation becomes prohibitively expensive and may exceed memory limits. 
\textbf{Therefore we resort to a numerical verification, which is presented in Appendix~\ref{app:discriminant}. }
The numerical experiments, performed with MATLAB on a sample of \(10^6\) random points in the region \(\gamma\in[1,3]\), \(M\in(-1,1)\), show that \(\Delta_{\mathrm{2nd}}\) is always positive to within machine precision, and the only zeros occur on the boundary \(M=-1\). 
\textbf{This provides strong evidence that \(\Delta_{\mathrm{2nd}}\ge 0\) throughout the interior of the domain, implying that the three eigenvalues are real. }

\subsection{Main results for AUSM}
\begin{theorem}[Eigenvalues of the linear AUSM splitting]\label{thm-AUSM-lin-eig}
For the linear AUSM splitting, under the conditions $\gamma>1$, $|M|<1$, and $a>0$, the Jacobian matrix $\partial F^+/\partial U$ has eigenvalues that are not all of the same sign. Specifically, the coefficients $S_{\mathrm{lin}}^{+} $ and $D_{\mathrm{lin}}^{+}$ change sign over the interval $(-1,1)$, so the eigenvalues cannot all be non-negative (nor all non-positive).
\end{theorem}

\begin{proof}
From Lemmas \ref{lemma-AUSM-Slin-sign-change} and \ref{lemma-AUSM-Dlin-sign-change}, both $S_{\mathrm{lin}}^{+} $ and $D_{\mathrm{lin}}^{+}$ take both positive and negative values for $M\in(-1,1)$. Since $S_{\mathrm{lin}}^{+}  = \mu_1\mu_2+\mu_1\mu_3+\mu_2\mu_3$ and $D_{\mathrm{lin}}^{+} = \mu_1\mu_2\mu_3$, if all eigenvalues were of the same sign, $S_{\mathrm{lin}}^{+} $ and $D_{\mathrm{lin}}^{+}$ would be non-negative (for all non-negative eigenvalues) or non-positive (for all non-positive eigenvalues). The sign changes observed contradict this, proving the theorem.
\end{proof}

\begin{proposition}[Eigenvalues of the second-order AUSM splitting]\label{thm-AUSM-2nd-eig}
\label{prop:ausm2}
For the second-order AUSM splitting, under the conditions \(\gamma>1\), \(|M|<1\) and \(a>0\), all three eigenvalues of \(\partial F^+/\partial U\) are positive real numbers. 
\end{proposition}

\begin{proof}
The positivity of $T_{\mathrm{2nd}}^{+} $, $S_{\mathrm{2nd}}^{+} $, and $D_{\mathrm{2nd}}^{+}$ follows from Lemmas \ref{lemma-AUSM-T2nd-positive}, \ref{lemma-AUSM-S2nd-positive}, and \ref{lemma-AUSM-D2nd-positive}, respectively. The statement about the discriminant is supported by the numerical experiments reported in Appendix \ref{app:discriminant}.
According to Viète's theorem, three eigenvalues are all positive. 
\end{proof}

\section{Conclusions}\label{sec-Conclusion}
We have provided a detailed mathematical proof that the Van Leer flux-vector splitting satisfies the intended eigenvalue sign condition: $\partial F^+/\partial U$ has one zero eigenvalue and two positive real eigenvalues for all physically relevant parameters. 
This fills a gap in the original presentation. For the AUSM family, we have shown that the linear pressure splitting does \textbf{not} possess this property, whereas a second-order pressure splitting yields coefficients that are all positive. 
Numerical evidence indicates that the discriminant of the cubic for the second-order pressure splitting is non-negative, so that its eigenvalues are real and positive. 
Our work provides a solid foundation for understanding the spectral behaviour of two widely used flux-vector splitting schemes.

\appendix
\section{Numerical verification of the discriminant for the Van Leer splitting and AUSM with second-order pressure splitting}
\label{app:discriminant}
In this appendix we present numerical evidence that the discriminant $\Delta_{\mathrm{vl}}(\gamma,M)$ and $\Delta_{\mathrm{2nd}}(\gamma,M)$ of the characteristic equations for the Van Leer splitting and second-order AUSM splitting are both non-negative throughout the parameter domain $\gamma\in(1,3]$, $M\in(-1,1)$. The two discriminants are given by
\begin{align*}
\Delta_{\mathrm{vl}} &= \left(T_{\mathrm{vl}}^{+}\right)^2 - 4S_{\mathrm{vl}}^{+}, \\
\Delta_{\mathrm{2nd}} &= 18 T_{\mathrm{2nd}}^{+}  S_{\mathrm{2nd}}^{+}  D_{\mathrm{2nd}}^{+} - 4 \left(T_{\mathrm{2nd}}^{+}\right)^3 D_{\mathrm{2nd}}^{+} + \left(T_{\mathrm{2nd}}^{+}\right)^2 \left(S_{\mathrm{2nd}}^{+}\right)^2 - 4 \left(S_{\mathrm{2nd}}^{+}\right)^3 - 27 \left(D_{\mathrm{2nd}}^{+}\right)^2,  	
\end{align*}
where $T_{\mathrm{vl}}^{+}$, $T_{\mathrm{vl}}^{+}$ are defined in Eq.\eqref{eq-VanLeer-T}\eqref{eq-VanLeer-S} while $T_{\mathrm{2nd}}^{+} $, $S_{\mathrm{2nd}}^{+} $, $D_{\mathrm{2nd}}^{+}$ in Eq.\eqref{eq-AUSM-2ndP-T}\eqref{eq-AUSM-2ndP-S}\eqref{eq-AUSM-2ndP-D}. 
We performed a systematic numerical scan over a fine grid in the $(\gamma,M)$ plane, as well as random sampling, and evaluated $\Delta_{\mathrm{vl}}$ and $\Delta_{\mathrm{2nd}}$ at each point. 
Note: According to Eq.\eqref{eq-VanLeer-discriminant},  the sign of $\Delta_{\mathrm{vl}}$ is determined solely by $H(\gamma,M)$. 
Therefore, in the numerical verification of Van Leer splitting, it suffices to sample $H(\gamma,M)$ directly to check its non-negativity. 
Additionally, in all numerical experiments, $a=1$ is taken without loss of generality, as it merely scales the discriminants positively. 
All numerical experiments were carried out using MATLAB. 
\par\noindent
\begin{itemize}
\item \textbf{Numerical verification of the discriminant for the Van Leer splitting: } 
\\
A random sample of $10^6$ points with $\gamma\in[1,3]$, $M\in(-1,1)$ gave no negative values (tolerance $10^{-12}$). 
The smallest observed value of $H$ was approximately $6.46\times10^{1}$, and a numerical optimisation located the minimum at $(\gamma,M)=(1.000,1.000)$ with $H=64.0$, which lies exactly on the boundary $M=1$. 
This confirms that $H$ is strictly positive in the interior of the domain. Figure~\ref{Fig.VanLeer-discriminant-H} shows the three-dimensional surface and a filled contour plot of $H$, visually confirming its positivity. 
\begin{figure}[htbp]
  \begin{center}
    \begin{minipage}{0.49\linewidth}
      \centerline{\includegraphics[width=1\linewidth]{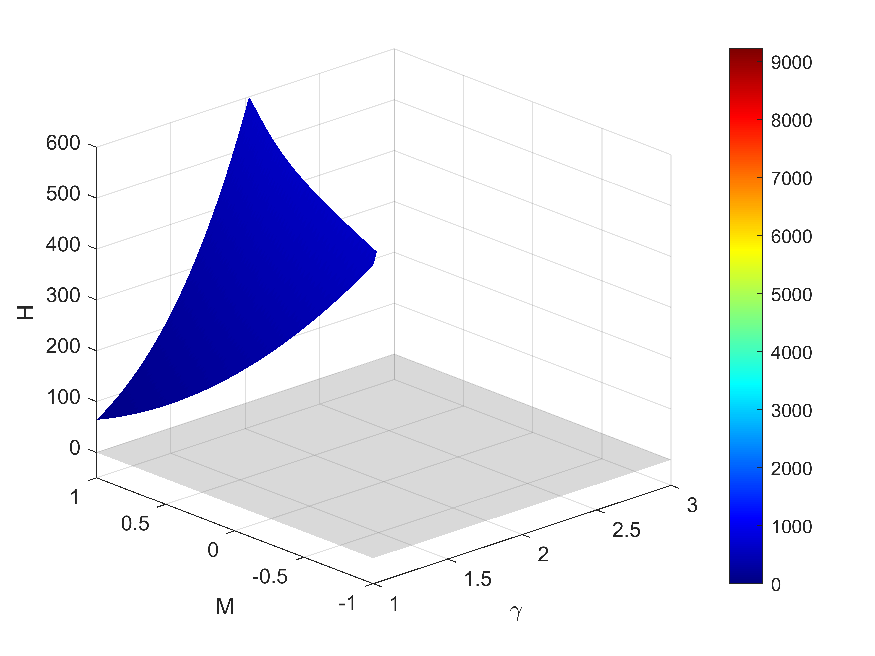}}
    \end{minipage}
    \hfill
    \begin{minipage}{0.49\linewidth}
      \centerline{\includegraphics[width=1\linewidth]{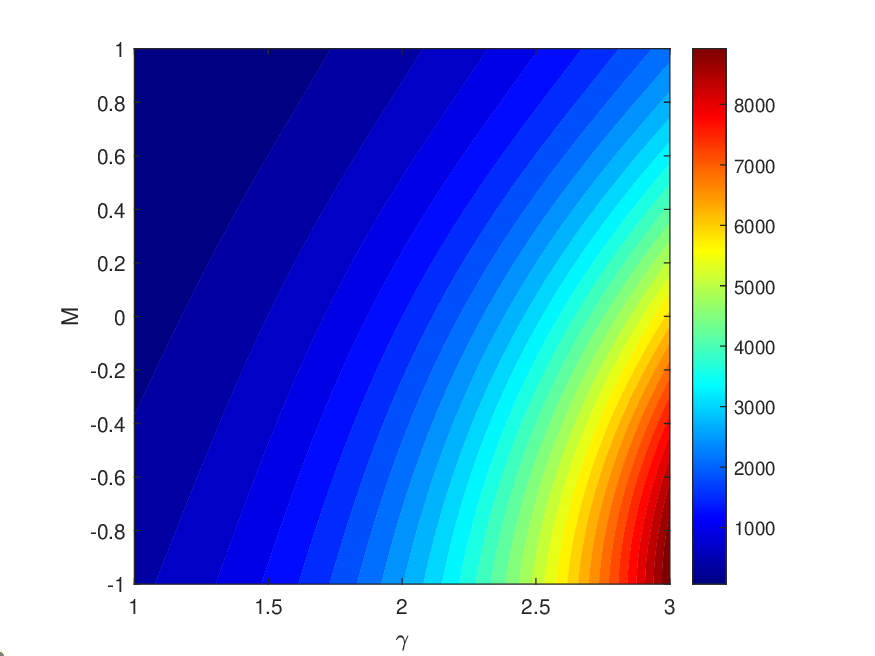}}
    \end{minipage}
    \vfill
    \begin{minipage}{0.49\linewidth}
      \centerline{\tiny{(a)\ 3D surface of $H(\gamma,M)$ with zero plane}}
    \end{minipage}
    \hfill
    \begin{minipage}{0.49\linewidth}
      \centerline{\tiny{(b)\ Filled contour plot of $H(\gamma,M)$}}
    \end{minipage}
  \end{center}
\caption{Numerical verification of $H(\gamma,M) \geq 0$ for Van Leer Splitting. }
\label{Fig.VanLeer-discriminant-H}
\end{figure}

\item \textbf{Numerical verification of the discriminant for AUSM with second-order pressure splitting: }
\\
Again setting $a=1$, we generated $10^6$ random points in the same parameter domain. All computed values of $\Delta_{\mathrm{2nd}}$ were positive to within numerical accuracy; the smallest value observed was $1.30\times10^{-42}$, which is essentially zero at the level of machine precision. An optimisation search found the minimum to be exactly $0$ at $(\gamma\approx2.114,\,M=-1.000)$, i.e. on the boundary $M=-1$. Hence $\Delta$ is strictly positive in the interior. 
Figure~\ref{Fig.AUSM-2ndP-discriminant} displays the three-dimensional surface and the corresponding contour plot, clearly showing that $\Delta_{\mathrm{2nd}}$ stays above zero. 
\begin{figure}[htbp]
  \begin{center}
    \begin{minipage}{0.49\linewidth}
      \centerline{\includegraphics[width=1\linewidth]{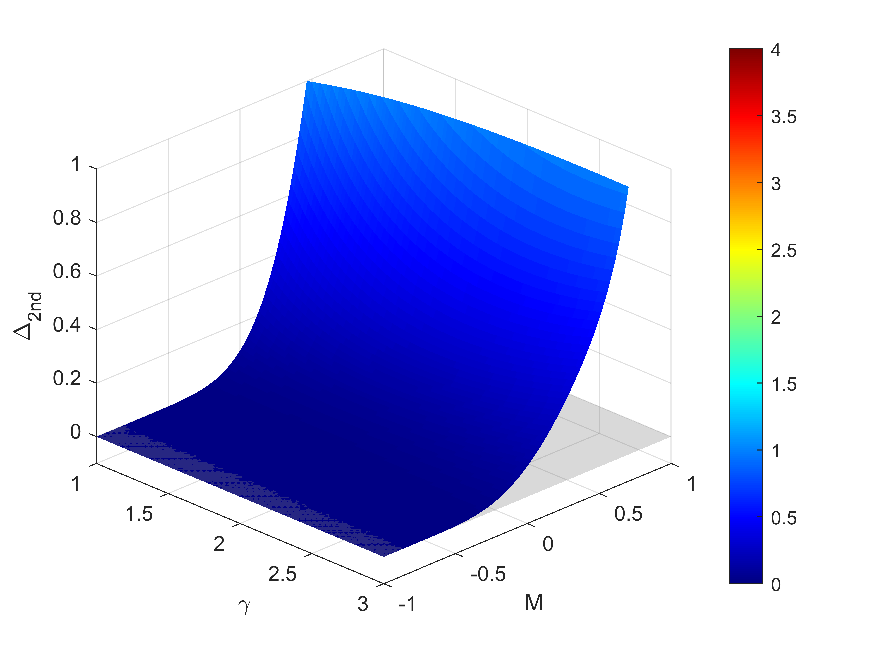}}
    \end{minipage}
    \hfill
    \begin{minipage}{0.49\linewidth}
      \centerline{\includegraphics[width=1\linewidth]{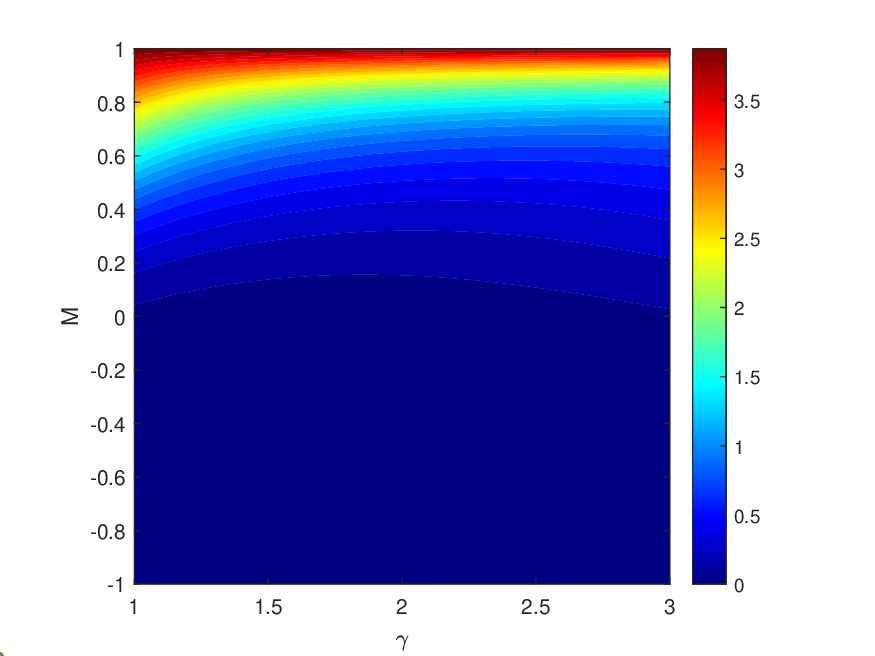}}
    \end{minipage}
    \vfill
    \begin{minipage}{0.49\linewidth}
      \centerline{\tiny{(a)\ 3D surface of $\Delta_{\mathrm{2nd}}(\gamma,M)$ with zero plane}}
    \end{minipage}
    \hfill
    \begin{minipage}{0.49\linewidth}
      \centerline{\tiny{(b)\ Filled contour plot of $\Delta_{\mathrm{2nd}}(\gamma,M)$}}
    \end{minipage}
  \end{center}
\caption{Numerical verification of $\Delta_{\mathrm{2nd}}(\gamma,M) \geq 0$ for AUSM with second-order pressure splitting. }
\label{Fig.AUSM-2ndP-discriminant}
\end{figure}
\end{itemize}
\par
The numerical experiments strongly support the theoretical conclusion that both discriminants are non-negative throughout the region of interest, implying that the corresponding eigenvalues are real.

\section{Explicit Jacobian matrices}
\label{app:jacobians}

For completeness, we list here the explicit expressions of the Jacobian matrices \(\partial F^+/\partial U\) in conservative variables for the three splittings studied in this paper. 
% In each case the matrix is obtained by transforming the primitive-variable Jacobian \(\partial F^+/\partial W\) with the inverse of \(\partial U/\partial W\), i.e.
% \[
% \frac{\partial F^+}{\partial U} = \frac{\partial F^+}{\partial W}\left(\frac{\partial U}{\partial W}\right)^{-1}.
% \]
All entries are functions of \(\gamma\), \(a\) and \(M\) only (the density \(\rho\) cancels out during the transformation). The expressions have been simplified with the aid of the computer algebra system Mathematica.  
As in the main text, we only present the matrices for the positive part \(F^+\); the matrices for \(F^-\) can be obtained by symmetry. 

\subsection{Van Leer splitting}
\label{app:vanleer}
\[
J^+_{\mathrm{vl},U}=\frac{\partial F^+_{\mathrm{vl}}}{\partial U} = [A_{ij}]_{i,j=1,2,3},
\]
with
{\scriptsize
\begin{align}
A_{11}&=-\frac{1}{16} a \left(M^2-1\right) \left((\gamma -1) \gamma  M^2+2\right), \notag \\ 
A_{12}&=\frac{1}{8} \left((\gamma -1) \gamma  M^3+\left(-\gamma ^2+\gamma +4\right) M+4\right), \notag \\ 
A_{13}&=-\frac{(\gamma -1) \gamma  (M-1) (M+1)}{8 a}, \notag \\
A_{21}&=-\frac{a^2 M (M+1) \left(2 (\gamma +3)+(\gamma -1)^2 \gamma  M^3-(\gamma -1)^2 \gamma  M^2-2 \left(2 \gamma ^2-5 \gamma +3\right) M\right)}{16 \gamma }, \notag \\ 
A_{22}&=\frac{a \left(2 (\gamma +3)+(\gamma -1)^2 \gamma  M^4-\left(\gamma ^3+2 \gamma ^2-9 \gamma +6\right) M^2-4 (\gamma -3) \gamma  M\right)}{8 \gamma }, \notag \\ 
A_{23}&=-\frac{1}{8} (\gamma -1) (M+1) \left((\gamma -1) M^2-\gamma  M+M-4\right), \notag \\
A_{31}&=-\frac{a^3 (M+1) \left((\gamma -1)^3 \gamma  M^5-(\gamma -1)^3 \gamma  M^4+\left(-8 \gamma ^3+22 \gamma ^2-24 \gamma +10\right) M^3+\left(-6 \gamma ^2+32 \gamma -26\right) M^2+8 (2 \gamma +1) M+8\right)}{32 \left(\gamma ^2-1\right)}, \notag \\ 
A_{32}&=\frac{a^2 (M+1) \left(8 (\gamma +1)+(\gamma -1)^3 \gamma  M^4-(\gamma -1)^3 \gamma  M^3-4 \left(2 \gamma ^3-5 \gamma ^2+5 \gamma -2\right) M^2-4 \left(2 \gamma ^2-7 \gamma +5\right) M\right)}{16 \left(\gamma ^2-1\right)}, \notag \\
A_{33}&=-\frac{a \gamma  (M+1) \left((\gamma -1)^2 M^3-(\gamma -1)^2 M^2+(4-8 \gamma ) M-12\right)}{16 (\gamma +1)}. \notag
\end{align}
}

\subsection{AUSM (Liou-Steffen) with linear pressure splitting}
\label{app:ausm-linear}
\[
J^+_{\mathrm{lin},U}=\frac{\partial F^+_{\mathrm{lin}}}{\partial U} = [B_{ij}]_{i,j=1,2,3},
\]
with
\begin{align}
B_{11}&=-\frac{1}{16} a \left(M^2-1\right) \left((\gamma -1) \gamma  M^2+2\right), \notag \\ 
B_{12}&=\frac{1}{8} \left((\gamma -1) \gamma  M^3+\left(-\gamma ^2+\gamma +4\right) M+4\right), \notag \\ 
B_{13}&=-\frac{(\gamma -1) \gamma  (M-1) (M+1)}{8 a}, \notag \\
B_{21}&=-\frac{a^2 M \left(-\gamma ^2 M \left(M^3+M+4\right)+\gamma ^3 M^2 \left(M^2-1\right)+2 \gamma  \left(4 M^2+6 M+1\right)+4\right)}{16 \gamma }, \notag \\ 
B_{22}&=\frac{a \left(-\gamma ^2 M \left(M^3+M+4\right)+\gamma ^3 M^2 \left(M^2-1\right)+2 \gamma  \left(4 M^2+6 M+1\right)+4\right)}{8 \gamma }, \notag \\ 
B_{23}&=-\frac{1}{8} (\gamma -1) \left(\gamma  M^3-(\gamma +2) M-4\right), \notag \\
B_{31}&=-\frac{a^3 (M+1) \left((\gamma -1)^2 \gamma  M^5-(\gamma -1)^2 \gamma  M^4-2 \left(\gamma ^2-6 \gamma +5\right) M^3-6 (\gamma -1)^2 M^2+12 M+4\right)}{32 (\gamma -1)}, \notag \\ 
B_{32}&=\frac{1}{16} a^2 (M+1) \left(\frac{8}{\gamma -1}+(\gamma -1) \gamma  M^4-(\gamma -1) \gamma  M^3-2 (\gamma -4) M^2+(4-6 \gamma ) M\right), \notag \\ 
B_{33}&=\frac{1}{16} a \gamma  \left(-\left((\gamma -1) M^4\right)+(\gamma +1) M^2+8 M+6\right). \notag
\end{align}

\subsection{AUSM (Liou-Steffen) with second-order pressure splitting}
\label{app:ausm-second}
\[
J^+_{\mathrm{2nd},U}=\frac{\partial F^+_{\mathrm{2nd}}}{\partial U} = [C_{ij}]_{i,j=1,2,3},
\]
with
\begin{align}
C_{11}&=-\frac{1}{16} a \left(M^2-1\right) \left((\gamma -1) \gamma  M^2+2\right), \notag \\ 
C_{12}&=\frac{1}{8} \left((\gamma -1) \gamma  M^3+\left(-\gamma ^2+\gamma +4\right) M+4\right), \notag \\ 
C_{13}&=-\frac{(\gamma -1) \gamma  (M-1) (M+1)}{8 a}, \notag \\
C_{21}&=-\frac{a^2 M (M+1) \left(2 (\gamma +3)+(\gamma -1)^2 \gamma  M^3-(\gamma -1)^2 \gamma  M^2-2 \left(2 \gamma ^2-5 \gamma +3\right) M\right)}{16 \gamma }, \notag \\ 
C_{22}&=\frac{a \left(2 (\gamma +3)+(\gamma -1)^2 \gamma  M^4-\left(\gamma ^3+2 \gamma ^2-9 \gamma +6\right) M^2-4 (\gamma -3) \gamma  M\right)}{8 \gamma }, \notag \\ 
C_{23}&=-\frac{1}{8} (\gamma -1) (M+1) \left((\gamma -1) M^2-\gamma  M+M-4\right), \notag \\
C_{31}&=-\frac{a^3 (M+1) \left((\gamma -1)^2 \gamma  M^5-(\gamma -1)^2 \gamma  M^4-2 \left(\gamma ^2-6 \gamma +5\right) M^3-6 (\gamma -1)^2 M^2+12 M+4\right)}{32 (\gamma -1)}, \notag \\ 
C_{32}&=\frac{1}{16} a^2 (M+1) \left(\frac{8}{\gamma -1}+(\gamma -1) \gamma  M^4-(\gamma -1) \gamma  M^3-2 (\gamma -4) M^2+(4-6 \gamma ) M\right), \notag \\ 
C_{33}&=\frac{1}{16} a \gamma  \left(-\left((\gamma -1) M^4\right)+(\gamma +1) M^2+8 M+6\right). \notag
\end{align}

\section{Sturm sequence proof of the non-negativity of the Van Leer discriminant}
\label{app:sturm}

In this appendix we rigorously prove, by means of Sturm's theorem, that the discriminant $\Delta_{\mathrm{vl}}$ of the Van Leer splitting is non-negative for $\gamma\in(1,3)$ and $M\in(-1,1)$. Consequently the quadratic factor $\mu^2-T_{\mathrm{vl}}^{+}\mu+S_{\mathrm{vl}}^{+}$ has two real roots.

From the main text,
\[
\Delta_{\mathrm{vl}}=\frac{a^{2}(M+1)^{2}}{64\gamma^{2}(\gamma+1)^{2}}\,H(\gamma,M),
\]
where $H(\gamma,M)$ is the following polynomial of degree six in $M$ (coefficients depend on $\gamma$):
\[
\begin{aligned}
H(\gamma,M) &= (\gamma-1)^{2}\gamma^{2}M^{6}+(-2\gamma^{4}+4\gamma^{3}-2\gamma^{2})M^{5}
+(-5\gamma^{4}-2\gamma^{3}+19\gamma^{2}-12\gamma)M^{4}\\
&\quad+(20\gamma^{4}-44\gamma^{2}+24\gamma)M^{3}
+(-13\gamma^{4}+26\gamma^{3}+39\gamma^{2}-24\gamma+36)M^{2}\\
&\quad+(-42\gamma^{4}-20\gamma^{3}-50\gamma^{2}-72\gamma-72)M
+(57\gamma^{4}+26\gamma^{3}+53\gamma^{2}+84\gamma+36).
\end{aligned}
\]
Since $a^{2}>0$ and the denominator is positive, $\Delta_{\mathrm{vl}}\ge0$ is equivalent to $H(\gamma,M)\ge0$. One easily checks $H(\gamma,-1)>0$ and $H(\gamma,1)>0$; therefore it suffices to prove that $H$ has no real root in $(-1,1)$. Together with the sign at the endpoints this yields $H>0$ in $(-1,1)$.

\subsection{Construction of the Sturm sequence}
\label{sec:sturm-construction}

Treating $H$ as a polynomial in $M$, we construct its Sturm sequence $\{p_{0},p_{1},\dots,p_{6}\}$ ($p_{0}=H$) by repeated polynomial division (taking the negative remainder). The calculations were performed with the computer algebra system Mathematica. The expressions below are given with all positive denominators omitted, as they do not affect the sign.
{\footnotesize
\begin{align*}
p_{0}(\gamma,M) &= H(\gamma,M),\\[2mm]
p_{1}(\gamma,M) &= 2\Big[-21\gamma^{4}-10\gamma^{3}-25\gamma^{2}-36\gamma
+3(\gamma-1)^{2}\gamma^{2}M^{5}-5(\gamma-1)^{2}\gamma^{2}M^{4} 
-2(\gamma-1)^{2}\gamma(5\gamma+12)M^{3} \\
&\quad+6\gamma(5\gamma^{3}-11\gamma+6)M^{2}
+(-13\gamma^{4}+26\gamma^{3}+39\gamma^{2}-24\gamma+36)M-36\Big],\\[2mm]
p_{2}(\gamma,M) &= \frac{4}{9}\Big[-123\gamma^{4}-56\gamma^{3}-113\gamma^{2}-180\gamma
+(\gamma-1)^{2}\gamma(5\gamma+9)M^{4}+\gamma(-20\gamma^{3}+\gamma^{2}+40\gamma-21)M^{3}\\
&\quad+3(4\gamma^{4}-13\gamma^{3}-14\gamma^{2}+9\gamma-18)M^{2}
+(82\gamma^{4}+31\gamma^{3}+84\gamma^{2}+141\gamma+126)M-72\Big],\\[2mm]
p_{3}(\gamma,M) &= -\frac{18}{(\gamma-1)(5\gamma+9)^{2}}\Big[420\gamma^{7}+544\gamma^{6}+303\gamma^{5}+971\gamma^{4}
549\gamma^{3}-63\gamma^{2}+216\gamma\\
&\quad+(30\gamma^{6}+44\gamma^{5}+45\gamma^{4}+115\gamma^{3}-81\gamma^{2}-207\gamma+54)\gamma M^{3}\\
&\quad+(-100\gamma^{6}+158\gamma^{5}+199\gamma^{4}-911\gamma^{3}+195\gamma^{2}+1089\gamma-54)\gamma M^{2}\\
&\quad-3(50\gamma^{7}+74\gamma^{6}+29\gamma^{5}+81\gamma^{4}+409\gamma^{3}+249\gamma^{2}-72\gamma+108)M
+324\Big],\\[2mm]
p_{4}(\gamma,M) &= -\frac{16\gamma(5\gamma+9)^{2}(\gamma^{2}-2\gamma-3)}{9\bigl(30\gamma^{5}+74\gamma^{4}+119\gamma^{3}+234\gamma^{2}+153\gamma-54\bigr)^{2}}
\times\Big[-687\gamma^{9}+3826\gamma^{8}+4339\gamma^{7}+3364\gamma^{6} \\
&\qquad+13555\gamma^{5}+5250\gamma^{4}+3261\gamma^{3}+5256\gamma^{2}-756\gamma+\Big(233\gamma^{9}-905\gamma^{8}+2315\gamma^{7}+4051\gamma^{6} \\
&\qquad-7051\gamma^{5}+4753\gamma^{4}+6321\gamma^{3}-3735\gamma^{2}+486\gamma-324\Big)M^{2}\\
&\qquad-(42\gamma^{9}+2019\gamma^{8}+3066\gamma^{7}+3681\gamma^{6}+4132\gamma^{5}+11329\gamma^{4}\\
&\qquad+6066\gamma^{3}-801\gamma^{2}+3078\gamma-2916)M-2592\Big],\\[2mm]
p_{5}(\gamma,M) &= -\frac{18\bigl(30\gamma^{5}+74\gamma^{4}+119\gamma^{3}+234\gamma^{2}+153\gamma-54\bigr)^{2}}{(5\gamma+9)^{2}\bigl(233\gamma^{9}-905\gamma^{8}+2315\gamma^{7}+4051\gamma^{6}-7051\gamma^{5}+4753\gamma^{4}+6321\gamma^{3}-3735\gamma^{2}+486\gamma-324\bigr)^{2}}\\
&\quad\times\Big[-8511\gamma^{14}-41054\gamma^{13}+401101\gamma^{12}-730596\gamma^{11}+725918\gamma^{10}+469916\gamma^{9}-978190\gamma^{8}+847552\gamma^{7}
&\qquad+374933\gamma^{6}-747390\gamma^{5}+657\gamma^{4}+242676\gamma^{3}\\
&\qquad+29484\gamma^{2}-73872\gamma+\Bigl(4741\gamma^{14}+7658\gamma^{13}-268187\gamma^{12}+672856\gamma^{11}-777322\gamma^{10}-186452\gamma^{9}+977506\gamma^{8} \\
&\qquad-899592\gamma^{7}-279895\gamma^{6}+562026\gamma^{5}+79497\gamma^{4}-82080\gamma^{3}-150660\gamma^{2}+89424\gamma-11664\Bigr)M+11664\Big], 
\end{align*}
}
% {\tiny
% \begin{align*}
% &p_{6}(\gamma) \\
% &= -\frac{64\gamma^{2}\bigl(1165\gamma^{11}-1263\gamma^{10}+1002\gamma^{9}+44520\gamma^{8}+42294\gamma^{7}-38490\gamma^{6}+34688\gamma^{5}+112596\gamma^{4}+7029\gamma^{3}-28431\gamma^{2}-162\gamma-2916\bigr)^{2}}{\bigl(-4741\gamma^{14}-7658\gamma^{13}+268187\gamma^{12}-672856\gamma^{11}+777322\gamma^{10}+186452\gamma^{9}-977506\gamma^{8}+899592\gamma^{7}+279895\gamma^{6}-562026\gamma^{5}-79497\gamma^{4}+82080\gamma^{3}+150660\gamma^{2}-89424\gamma+11664\bigr)^{2}}\\
% &\quad\times\Big[1191\gamma^{18}-184324\gamma^{17}-407081\gamma^{16}+9695572\gamma^{15}-40340409\gamma^{14}+86095716\gamma^{13}-95230669\gamma^{12}\\
% &\qquad+30188892\gamma^{11}+82420025\gamma^{10}-136322892\gamma^{9}+77078389\gamma^{8}+14495532\gamma^{7}-52620435\gamma^{6}\\
% &\qquad+26160732\gamma^{5}-505359\gamma^{4}+4715172\gamma^{3}-7829460\gamma^{2}+2904336\gamma-314928\Big].
% \end{align*}
% }
\begin{figure}[htbp]
  \begin{center}
      \centerline{\includegraphics[width=1.25\linewidth]{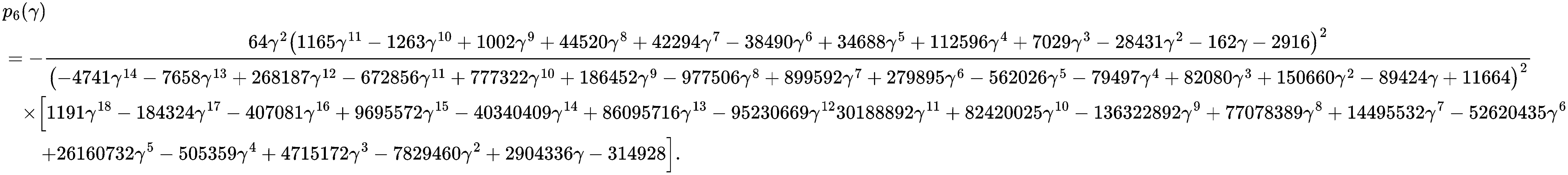}}
  \end{center}
% \caption{\tiny Schematic Diagram of 1D-iALE-CDG Algorithm. }
\label{Fig.VanLeer_Sturm_P5_M_1}
\end{figure}
\hfill \\
Note that $p_{6}$ no longer depends on $M$; it is a function of $\gamma$ alone.

\subsection{Sign of each $p_i$ at the endpoints $M=\pm1$}
\label{sec:sign-endpoints}

For a fixed $\gamma\in(1,3)$ we determine the sign of $p_i(\gamma,1)$ and $p_i(\gamma,-1)$. All denominators appearing in the $p_i$ are positive for $\gamma>1$, so the sign is that of the corresponding numerator.

\subsubsection*{$\bullet~p_0$}
Direct substitution gives
\[
p_{0}(\gamma,1)=16\gamma^{2}(\gamma+1)^{2}>0,\qquad 
p_{0}(\gamma,-1)=64\gamma^{4}+64\gamma^{3}+208\gamma^{2}+96\gamma+144>0.
\]

\subsubsection*{$\bullet~p_1$}
\begin{itemize}
  \item At $M=1$, 
\[
p_{1}(\gamma,1)=-32\gamma\bigl(\gamma^{3}-\gamma^{2}+\gamma+3\bigr).
\]
Set $q(\gamma)=\gamma^{3}-\gamma^{2}+\gamma+3$; then $q'(\gamma)=3\gamma^{2}-2\gamma+1$ has discriminant $-8<0$, hence $q'>0$ and $q$ is strictly increasing. Since $q(1)=4>0$, we have $q(\gamma)>0$ for all $\gamma>1$, and therefore $p_{1}(\gamma,1)<0$.
  
  \item At $M=-1$, 
for $p_{1}(\gamma,-1)=16\bigl(3\gamma^{4}-2\gamma^{3}-22\gamma^{2}+6\gamma-9\bigr)$ denote $r(\gamma)=3\gamma^{4}-2\gamma^{3}-22\gamma^{2}+6\gamma-9$. One finds $r(1)=-24<0$, $r(3)=0$. The derivatives are
$r'(\gamma)=12\gamma^{3}-6\gamma^{2}-44\gamma+6$,
$r''(\gamma)=36\gamma^{2}-12\gamma-44$.
$r''$ has a unique zero in $(1,3)$ with $r''(1)<0$, $r''(3)>0$; consequently $r'$ first decreases then increases. Because $r'(1)=-32<0$ and $r'(3)=144>0$, there is a unique $\gamma^{*}_{1}\in(1,3)$ where $r'=0$, and $r$ decreases on $(1,\gamma^{*}_{1})$ and increases on $(\gamma^{*}_{1},3)$. With $r(1)<0$ and $r(3)=0$ we obtain $r(\gamma)<0$ for all $\gamma\in(1,3)$ (the only zero is at $\gamma=3$). Hence $p_{1}(\gamma,-1)<0$. 
\end{itemize}

\subsubsection*{$\bullet~p_2$}
\[
p_{2}(\gamma,1)=-\frac{16}{9}\gamma\bigl(11\gamma^{3}+16\gamma^{2}+11\gamma+6\bigr)<0,\qquad
p_{2}(\gamma,-1)=-\frac{16}{9}\bigl(42\gamma^{4}+32\gamma^{3}+73\gamma^{2}+66\gamma+63\bigr)<0,
\]
because all polynomials inside the parentheses have positive coefficients.

\subsubsection*{$\bullet~p_3$}
\begin{itemize}
  \item At $M=1$, write
\[
p_{3}(\gamma,1)=-\frac{72\gamma\,f(\gamma)}{(\gamma-1)(5\gamma+9)^{2}},\quad 
f(\gamma)=50\gamma^{6}+131\gamma^{5}+115\gamma^{4}-17\gamma^{3}-141\gamma^{2}+18\gamma+108.
\]
The denominator is positive. We have $f(1)=264>0$ and
\[
f'(\gamma)=300\gamma^{5}+655\gamma^{4}+460\gamma^{3}-51\gamma^{2}-282\gamma+18.
\]
For $\gamma\ge1$, using $\gamma^{2}\le\gamma^{3}$ and $\gamma\le\gamma^{3}$,
\[
-51\gamma^{2}-282\gamma\ge-51\gamma^{3}-282\gamma^{3}=-333\gamma^{3},
\]
so $f'(\gamma)\ge300\gamma^{5}+655\gamma^{4}+460\gamma^{3}-333\gamma^{3}+18>0$. Thus $f$ is strictly increasing and $f>0$; consequently $p_{3}(\gamma,1)<0$.

\item At $M=-1$, similarly,
\[
p_{3}(\gamma,-1)=-\frac{72\,g(\gamma)}{(\gamma-1)(5\gamma+9)^{2}},\quad 
g(\gamma)=110\gamma^{7}+220\gamma^{6}+136\gamma^{5}+47\gamma^{4}+513\gamma^{3}+495\gamma^{2}-27\gamma+162.
\]
$g(1)=1656>0$ and
\[
g'(\gamma)=770\gamma^{6}+1320\gamma^{5}+680\gamma^{4}+188\gamma^{3}+1539\gamma^{2}+990\gamma-27.
\]
For $\gamma\ge1$ each positive term is at least its value at $\gamma=1$, therefore
\[
g'(\gamma)\ge770+1320+680+188+1539+990-27=5460>0.
\]
Hence $g$ is strictly increasing, $g>0$, and $p_{3}(\gamma,-1)<0$.
\end{itemize}

\subsubsection*{$\bullet~p_4$}
\begin{itemize}
  \item At $M=1$, 
for $p_{4}(\gamma,1)$ the factor in front of $A(\gamma)$ is positive except for $(\gamma-3)<0$, and
\[
A(\gamma)=248\gamma^{7}-699\gamma^{6}-1095\gamma^{5}-772\gamma^{4}-414\gamma^{3}+1077\gamma^{2}-2835\gamma+1674.
\]
From its derivatives one obtains the following sign pattern (detailed values are collected in Table~\ref{tab:A-derivatives}):
For \(\gamma\ge 1\), \(A^{(6)}(\gamma)=1736\gamma-699>0\) on \([1,3]\). Thus \(A^{(5)}\) is strictly increasing on this interval. 
Note 
\(
A^{(5)}(1)<0,~
A^{(5)}(3)>0,
\) 
by the Intermediate Value Theorem, there exists a unique \(\xi_{5}\in(1,3)\) such that \(A^{(5)}(\xi_{5})=0\); moreover, \(A^{(5)}(\gamma)<0\) for \(\gamma\in[1,\xi_{5})\) and \(A^{(5)}(\gamma)>0\) for \(\gamma\in(\xi_{5},3]\).
Consequently, \(A^{(4)}\) decreases on \([1,\xi_{5}]\) and increases on \([\xi_{5},3]\). Its endpoint values are 
\(
A^{(4)}(1)<0,~
A^{(4)}(3)>0. 
\) 
Hence there exists a unique \(\xi_{4}\in(\xi_{5},3)\) with \(A^{(4)}(\xi_{4})=0\); \(A^{(4)}<0\) on \([1,\xi_{4})\) and \(>0\) on \((\xi_{4},3]\).
Next, \(A'''\) decreases on \([1,\xi_{4}]\) and increases on \([\xi_{4},3]\). We have 
\(
A'''(1)=<0,~
A'''(3)>0, 
\) 
so there exists a unique \(\xi_{3}\in(\xi_{4},3)\) with \(A'''(\xi_{3})=0\); \(A'''<0\) on \([1,\xi_{3})\) and \(>0\) on \((\xi_{3},3]\). 
Then \(A''\) decreases on \([1,\xi_{3}]\) and increases on \([\xi_{3},3]\). Its endpoint values are 
\(
A''(1)=<0,~
A''(3)=>0,
\) 
hence there exists a unique \(\xi_{2}\in(\xi_{3},3)\) with \(A''(\xi_{2})=0\); \(A''<0\) on \([1,\xi_{2})\) and \(>0\) on \((\xi_{2},3]\). 
Finally, \(A'\) decreases on \([1,\xi_{2}]\) and increases on \([\xi_{2},3]\). Its values at the endpoints are 
\(
A'(1)<0,~ A'(3)<0. 
\) 
Hence, \(A'(\gamma)<0\) for every \(\gamma\in[1,3]\). 
Therefore \(A\) is strictly decreasing on \([1,3]\). With \(A(1)<0\), we conclude \(A(\gamma)<0\) for all \(\gamma\in[1,3]\), 
implies \(p_4(\gamma,1)>0\) for every \(\gamma\in(1,3)\). 
\begin{table}[h!]
\centering
\caption{Values of \(A(\gamma)\) and its derivatives at the endpoints}
\label{tab:A-derivatives}
\begin{tabular}{r|ll}
\hline
 & \(\gamma=1\) & \(\gamma=3\) \\ \hline
\(A^{(6)}(\gamma)\) & \(746\,640\) & \(3\,246\,480\) \\
\(A^{(5)}(\gamma)\) & \(-9\,720\) & \(3\,983\,400\) \\
\(A^{(4)}(\gamma)\) & \(-193\,248\) & \(2\,947\,152\) \\
\(A'''(\gamma)\) & \(-118\,512\) & \(1\,304\,352\) \\
\(A''(\gamma)\) & \(-42\,048\) & \(152\,544\) \\
\(A'(\gamma)\) & \(-12\,944\) & \(-288\,000\) \\
\(A(\gamma)\) & \(-2\,816\) & \(-304\,128\) \\ \hline
\end{tabular}
\end{table}

\item At $M=-1$, 
for $p_{4}(\gamma,-1)$ the factor in front of $B(\gamma)$ is $\gamma^{2}-2\gamma-3<0$, and
\[
B(\gamma)=103\gamma^{9}-1235\gamma^{8}-2430\gamma^{7}-2774\gamma^{6}-2659\gamma^{5}-5333\gamma^{4}-3912\gamma^{3}-180\gamma^{2}-702\gamma+1458.
\]
The endpoint values of the derivatives of $B$ are listed in Table~\ref{tab:B-derivatives}. Analysing from the highest derivative downwards (similar to the treatment of $A$) shows that $B'<0$ on $[1,3]$ and $B(1)=-17664<0$; therefore $B(\gamma)<0$. The product of the two negatives gives $p_{4}(\gamma,-1)>0$.

\begin{table}[h!]
\centering
\caption{Values of the derivatives of $B(\gamma)$ at the endpoints}
\label{tab:B-derivatives}
\begin{tabular}{r|ll}
\hline
 & $\gamma=1$ & $\gamma=3$ \\ \hline
$B^{(8)}(\gamma)$ & $-12\,418\,560$ & $62\,334\,720$ \\
$B^{(7)}(\gamma)$ & $-43\,354\,080$ & $6\,562\,080$ \\
$B^{(6)}(\gamma)$ & $-32\,912\,640$ & $-94\,622\,400$ \\
$B^{(5)}(\gamma)$ & $-15\,181\,800$ & $-159\,355\,560$ \\
$B^{(4)}(\gamma)$ & $-5\,250\,240$ & $-157\,556\,496$ \\
$B^{\prime\prime\prime}(\gamma)$ & $-1\,517\,232$ & $-115\,156\,800$ \\
$B^{\prime\prime}(\gamma)$ & $-388\,032$ & $-67\,822\,848$ \\
$B^{\prime}(\gamma)$ & $-90\,032$ & $-33\,730\,560$ \\ 
$B(\gamma)$       & $-17664$    & $-14598144$ \\ \hline
\end{tabular}
\end{table}
\end{itemize}

\subsubsection*{$\bullet~p_5$}
\begin{itemize}
  \item At $M=-1$, for $p_{5}(\gamma,-1)$ the numerator is the polynomial
\begin{align*}
C(\gamma) =& 3313\gamma^{14}+12178\gamma^{13}-167322\gamma^{12}+350863\gamma^{11}-375810\gamma^{10}-164092\gamma^{9}+488924\gamma^{8}-436786\gamma^{7} \\
& -163707\gamma^{6}+327354\gamma^{5}+19710\gamma^{4}-81189\gamma^{3}-45036\gamma^{2}+40824\gamma-5832.
\end{align*}
Table~\ref{tab:C-derivatives} gives the endpoint values of its derivatives. Starting from $C^{(13)}$ (which is linear with $C^{(13)}(1)>0$, $C^{(13)}(3)>0$) and proceeding downwards, one finds that $C'<0$ on $[1,3]$ and $C(1)=-196608<0$; hence $C(\gamma)<0$ for all $\gamma\in(1,3)$. All other factors in $p_{5}(\gamma,-1)$ are positive, so $p_{5}(\gamma,-1)<0,~\forall \gamma\in(1,3)$. 
\begin{table}[h!]
\centering
\caption{Values of the derivatives of $C(\gamma)$ at the endpoints}
\label{tab:C-derivatives}
\begin{tabular}{r|ll}
\hline
 & $\gamma=1$ & $\gamma=3$ \\ \hline
$C^{(13)}(\gamma)$ & $364654338048000$ & $942297695539200$ \\
$C^{(12)}(\gamma)$ & $140095992960000$ & $1447048026547200$ \\
$C^{(11)}(\gamma)$ & $19911098592000$ & $1414507332268800$ \\
$C^{(10)}(\gamma)$ & $-2759150822400$ & $996008602176000$ \\
$C^{(9)}(\gamma)$  & $-2211997092480$ & $539008562090880$ \\
$C^{(8)}(\gamma)$  & $-693879943680$ & $232223479246080$ \\
$C^{(7)}(\gamma)$  & $-161262239040$ & $80729985792960$ \\
$C^{(6)}(\gamma)$  & $-31604832000$ & $22326801897600$ \\
$C^{(5)}(\gamma)$  & $-5513036160$ & $4523784364800$ \\
$C^{(4)}(\gamma)$  & $-875367936$ & $403344451584$ \\
$C^{\prime\prime\prime}(\gamma)$  & $-126127488$ & $-172475993088$ \\
$C^{\prime\prime}(\gamma)$  & $-16296448$ & $-120012005376$ \\
$C^{\prime}(\gamma)$  & $-1884160$ & $-47146696704$ \\ 
$C(\gamma)$  & $-196608$  & $-14714929152$ \\ \hline
\end{tabular}
\end{table}

\item At $M=1$, for $p_{5}(\gamma,1)$ the numerator is
\begin{align*}
D(\gamma)=& 1885\gamma^{13}+16698\gamma^{12}-66457\gamma^{11}+28870\gamma^{10}+25702\gamma^{9}-141732\gamma^{8}+342\gamma^{7}+26020\gamma^{6}-47519\gamma^{5} \\
& +92682\gamma^{4}-40077\gamma^{3}-80298\gamma^{2}+60588\gamma-7776.
\end{align*}
Its derivative data are collected in Table~\ref{tab:D-derivatives}. Analysing from the highest derivative shows that $D'$ changes sign (only once) from negative to positive, so $D$ first decreases then increases. With $D(1)=-131072<0$ and $D(3)=1401421824>0$, there exists a unique $\gamma^{*}_{5}\in(1,3)$ such that $D(\gamma^{*}_{5})=0$. Consequently
\[
p_{5}(\gamma,1)<0\ \text{for}\ \gamma\in(1,\gamma^{*}_{5}),\qquad 
p_{5}(\gamma,1)>0\ \text{for}\ \gamma\in(\gamma^{*}_{5},3).
\]
The graph of $D(\gamma)$ is shown in Figure~\ref{Fig.VanLeer_Sturm_P5_M_1_D}. 
\begin{table}[h!]
\centering
\caption{Values of the derivatives of $D(\gamma)$ at the endpoints}
\label{tab:D-derivatives}
\begin{tabular}{r|ll}
\hline
 & $\gamma=1$ & $\gamma=3$ \\ \hline
$D^{(13)}(\gamma)$ & $11737934208000$ & $11737934208000$ \\
$D^{(12)}(\gamma)$ & $19736302924800$ & $43212171340800$ \\
$D^{(11)}(\gamma)$ & $11214585043200$ & $74163059308800$ \\
$D^{(10)}(\gamma)$ & $3407519404800$ & $80959874284800$ \\
$D^{(9)}(\gamma)$  & $609856853760$ & $63994425788160$ \\
$D^{(8)}(\gamma)$  & $44950187520$ & $39320133834240$ \\
$D^{(7)}(\gamma)$  & $-11164487040$ & $19624584873600$ \\
$D^{(6)}(\gamma)$  & $-5585713920$ & $8189808076800$ \\
$D^{(5)}(\gamma)$  & $-1483180800$ & $2912811087360$ \\
$D^{(4)}(\gamma)$  & $-304326144$ & $892843978752$ \\
$D^{\prime\prime\prime}(\gamma)$  & $-52071168$ & $236581337088$ \\
$D^{\prime\prime}(\gamma)$  & $-7744512$ & $53714534400$ \\
$D^{\prime}(\gamma)$  & $-1048576$ & $10101915648$ \\ 
$D(\gamma)$  & $-131072$  & $1401421824$ \\ \hline
\end{tabular}
\end{table}
\begin{figure}[htbp]
  \begin{center}
      \centerline{\includegraphics[width=0.65\linewidth]{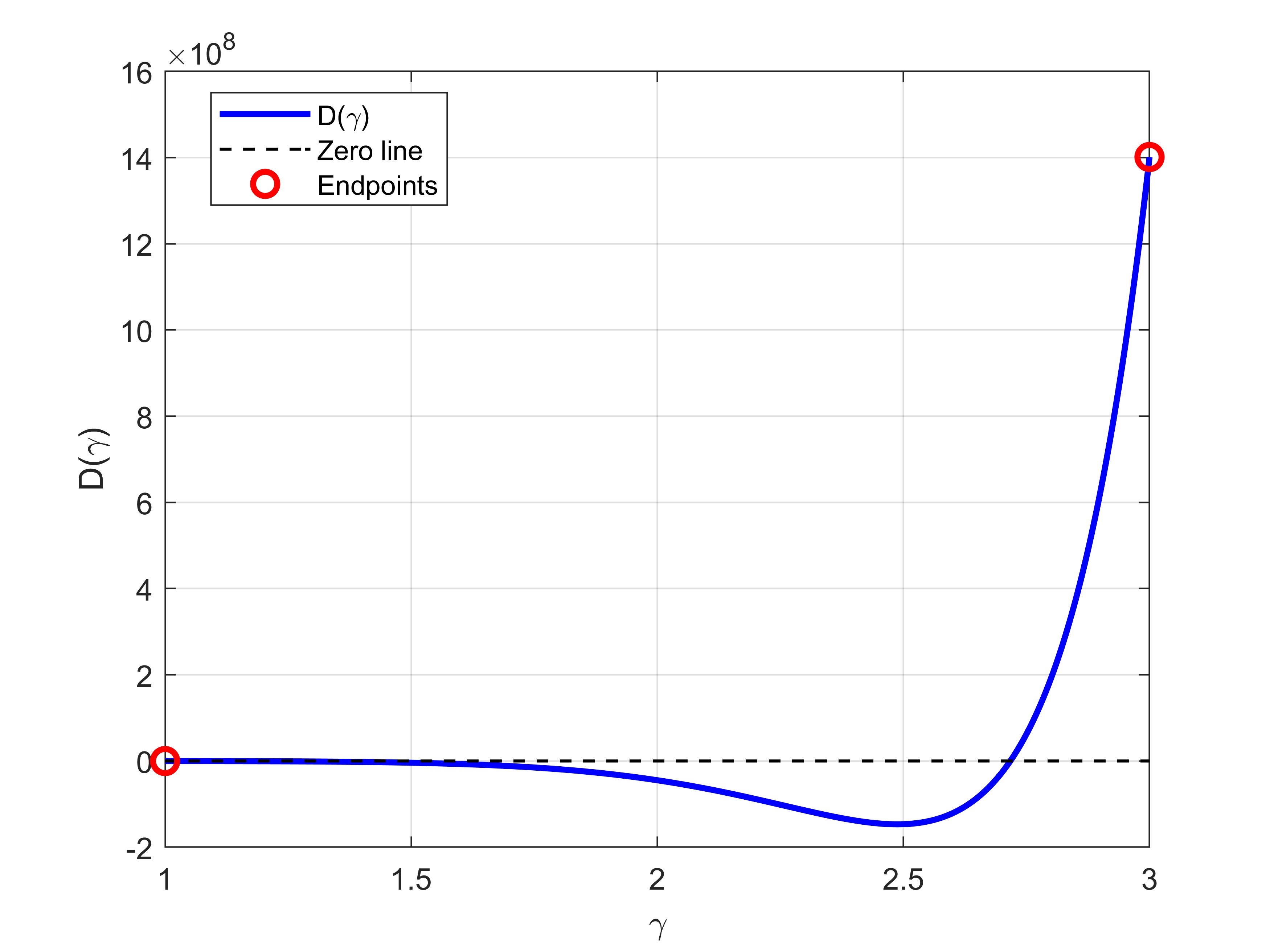}}
  \end{center}
\caption{\tiny Polynomial $D(\gamma)$ on $[1,3]$. }
\label{Fig.VanLeer_Sturm_P5_M_1_D}
\end{figure}
\end{itemize}

\subsubsection*{$\bullet~p_6$}
$p_6(\gamma)$ can be written as
\[
p_{6}(\gamma)=-\frac{64\gamma^{2}\,P(\gamma)^{2}}{R(\gamma)^{2}}\,Q(\gamma),
\]
where $P(\gamma)$ and $R(\gamma)$ are the polynomials appearing in the numerator and denominator of the first factor; they are clearly positive for $\gamma>1$. The sign of $p_{6}$ is therefore opposite to that of
\begin{align*}
Q(\gamma)=&1191\gamma^{18}-184324\gamma^{17}-407081\gamma^{16}+9695572\gamma^{15}-40340409\gamma^{14}+86095716\gamma^{13}-95230669\gamma^{12} \\
&+30188892\gamma^{11}+82420025\gamma^{10}-136322892\gamma^{9}+77078389\gamma^{8}+14495532\gamma^{7}-52620435\gamma^{6} \\
&+26160732\gamma^{5}-505359\gamma^{4}+4715172\gamma^{3}-7829460\gamma^{2}+2904336\gamma-314928.
\end{align*}
Table~\ref{tab:Q-derivatives} displays the endpoint values of the derivatives of $Q$. Starting from the constant $Q^{(18)}>0$, one finds that $Q^{\prime}$ is first positive then negative; hence $Q$ increases up to a unique maximum and then decreases. Because $Q(1)=Q(3)=0$, we obtain $Q(\gamma)>0$ for all $\gamma\in(1,3)$. Consequently $p_{6}(\gamma)<0$ on $(1,3)$ and $p_{6}(1)=p_{6}(3)=0$.

\begin{table}[h!]
\centering
\caption{Values of the derivatives of $Q(\gamma)$ at the endpoints}
\label{tab:Q-derivatives}
\begin{tabular}{r|ll}
\hline
 & $\gamma=1$ & $\gamma=3$ \\ \hline
% $Q^{(18)(\gamma)}$ & $7625227083522048000$ & $7625227083522048000$ \\
$Q^{(17)}(\gamma)$ & $-57936502412845056000$ & $-42686048245800960000$ \\
$Q^{(16)}(\gamma)$ & $-70266386185003008000$ & $-170888936843649024000$ \\
$Q^{(15)}(\gamma)$ & $-27348612810494976000$ & $-273587420561495040000$ \\
$Q^{(14)}(\gamma)$ & $-5706029171543500800$ & $-273101212323984844800$ \\
$Q^{(13)}(\gamma)$ & $-729102823242393600$ & $-197117842753614182400$ \\
$Q^{(12)}(\gamma)$ & $-45434704276684800$ & $-110996709244956057600$ \\
$Q^{(11)}(\gamma)$ & $5269511290982400$ & $-51078210726870835200$ \\
$Q^{(10)}(\gamma)$ & $2522668255027200$ & $-19815817769091993600$ \\
$Q^{(9)}(\gamma)$  & $595405723852800$ & $-6624699226522030080$ \\
$Q^{(8)}(\gamma)$  & $112374681108480$ & $-1938841892060774400$ \\
$Q^{(7)}(\gamma)$  & $18308704296960$ & $-502270624036823040$ \\
$Q^{(6)}(\gamma)$  & $2608995594240$ & $-115962154917396480$ \\
$Q^{(5)}(\gamma)$  & $323993088000$ & $-23911864155832320$ \\
$Q^{(4)}(\gamma)$  & $34432942080$ & $-4384097958887424$ \\
$Q^{\prime\prime\prime}(\gamma)$  & $3023659008$ & $-702783487475712$ \\
$Q^{\prime\prime}(\gamma)$  & $206831616$ & $-93943165353984$ \\
$Q^{\prime}(\gamma)$  & $9437184$ & $-8879408676864$ \\ 
$Q(\gamma)$  & $0$ & $0$ \\ \hline
\end{tabular}
\end{table}

\subsection{Sign changes and application of Sturm's theorem}
\label{sec:sturm-application}
According to Sturm's theorem, the number of distinct real roots of $H(\gamma,M)$ in $(-1,1)$ equals the difference of the numbers of sign changes of the sequence $\{p_{0},p_{1},\dots,p_{6}\}$ at $M=-1$ and $M=1$.
\par
For $\gamma\in(1,\gamma^{*}_{5})$ the signs at the endpoints are
\[
\begin{array}{c|c|c}
 & M=-1 & M=1 \\ \hline
p_{0} & + & + \\
p_{1} & - & - \\
p_{2} & - & - \\
p_{3} & - & - \\
p_{4} & + & + \\
p_{5} & - & - \\
p_{6} & - & - \\
\end{array}
\]
% (where a zero value is treated by its right-hand limit).  
\begin{itemize}
  \item At $M=-1$ the sign pattern $+,-,-,-,+,-,-$ contains three sign changes ($+\to-$, $-\to+$, $+\to-$). 
  \item At $M=1$ the pattern is the same, hence also three changes.
\end{itemize} 
\par 
For $\gamma\in(\gamma^{*}_{5},3)$ the sign of $p_{5}(\gamma,1)$ becomes $+$, while all other signs remain unchanged:
\[
\begin{array}{c|c|c}
 & M=-1 & M=1 \\ \hline
p_{0} & + & + \\
p_{1} & - & - \\
p_{2} & - & - \\
p_{3} & - & - \\
p_{4} & + & + \\
p_{5} & - & + \\
p_{6} & - & - \\
\end{array}
\]
\begin{itemize}
  \item At $M=-1$ the pattern is unchanged, still three changes.  
  \item At $M=1$ the pattern is $+,-,-,-,+,+,-$, which also exhibits three sign changes ($+\to-$, $-\to+$, $+\to-$).
\end{itemize} 
\par
Thus for every $\gamma\in(1,3)$ we have $V(-1)=V(1)=3$, and Sturm's theorem guarantees that $H(\gamma,M)$ has \textbf{no} root in $(-1,1)$.

\subsection{Conclusion}
\label{sec:conclusion}
Because $H(\gamma,-1)>0$, $H(\gamma,1)>0$ and there is no root inside $(-1,1)$, we obtain $H(\gamma,M)>0$ for all $M\in(-1,1)$ and all $\gamma \in (1,3)$. Hence $\Delta_{\mathrm{vl}}>0$ (except on the boundary $M=-1$), and the two non-zero eigenvalues of the Van Leer splitting are real.  
This completes the proof of Lemma~\ref{lemma-vl-discriminant-real}.  

\end{document}